\documentclass{amsart}
\usepackage{url}

\usepackage{amssymb, amsmath, amsthm}
\usepackage{hyperref}

\usepackage{appendix}

\usepackage{comment}

\newtheorem{theorem}{Theorem}[section]

\newtheorem{lemma}[theorem]{Lemma}

\newtheorem{proposition}[theorem]{Proposition}

\theoremstyle{definition}

\newtheorem{definition}[theorem]{Definition}
\newtheorem{remark}[theorem]{Remark}
\newtheorem{example}[theorem]{Example}

\newtheorem*{ack}{Acknowledgements}

\usepackage{amssymb,amsfonts,amsmath,amsopn,amstext,amscd,latexsym,color}
\usepackage[all]{xy}
\usepackage{mathtools}
\DeclarePairedDelimiter{\ceil}{\lceil}{\rceil}
\DeclarePairedDelimiter{\floor}{\lfloor}{\rfloor}

\usepackage[pdftex]{graphicx}

\xyoption{web}

\newcommand{\vphi}{{\varphi}}

\newcommand{\bbA}{{\mathbb A}}

\newcommand{\bbF}{{\mathbb F}}
\newcommand{\bbG}{{\mathbb G}}

\newcommand{\bbP}{{\mathbb P}}
\newcommand{\bbQ}{{\mathbb Q}}
\newcommand{\bbR}{{\mathbb R}}

\newcommand{\bbZ}{{\mathbb Z}}

\newcommand{\frp}{{\mathfrak p}}
\newcommand{\frq}{{\mathfrak q}}
\newcommand{\frr}{{\mathfrak r}}
\newcommand{\frs}{{\mathfrak s}}

\newcommand{\cA}{{\mathcal A}}
\newcommand{\cB}{{\mathcal B}}
\newcommand{\cC}{{\mathcal C}}

\newcommand{\cO}{{\mathcal O}}

\newcommand{\cT}{{\mathcal T}}

\newcommand{\et}{{\textnormal{\'et}}}

\usepackage[pdftex]{graphicx}

\DeclareMathOperator{\Gal}{Gal}
\DeclareMathOperator{\Br}{Br}

\DeclareMathOperator{\Aut}{Aut}
\DeclareMathOperator{\Hom}{Hom}
\DeclareMathOperator{\End}{End}

\DeclareMathOperator{\Kum}{Kum}
\DeclareMathOperator{\ord}{ord}

\DeclareMathOperator{\Res}{Res}
\DeclareMathOperator{\ev}{ev}

\author{Rachel Newton}
\address{Rachel Newton\\
Department of Mathematics\\
King's College London\\
Strand\\ 
London\\
WC2R 2LS\\
UK}
\email{rachel.newton@kcl.ac.uk}

\title[Brauer groups of products of CM elliptic curves]{Transcendental Brauer groups of products of CM elliptic curves}
\thanks{11G35 (primary), 11G05, 14K22, 11R37 (secondary)}

\begin{document}
\maketitle

\begin{abstract}
Let $L$ be a number field and let $E/L$ be an elliptic curve with complex multiplication by the ring of integers $\cO_K$ of an imaginary quadratic field $K$. We use class field theory and results of Skorobogatov and Zarhin to compute the transcendental part of the Brauer group of the abelian surface $E\times E$.
The results for the odd order torsion also apply to the Brauer group of the K3 surface $\Kum(E\times E)$. 
We describe explicitly the elliptic curves $E/\bbQ$ with complex multiplication by $\cO_K$ such that the Brauer group of $E\times E$ contains a transcendental element of odd order. We show that such an element gives rise to a Brauer-Manin obstruction to weak approximation on $\Kum(E\times E)$, while there is no obstruction coming from the algebraic part of the Brauer group.
\end{abstract}


\section{Introduction}

Let $X$ be a smooth, projective, geometrically irreducible variety over a number field $L$. In \cite{Manin}, Manin showed that the Brauer group of $X$ can obstruct the Hasse principle on $X$. Let $X(\bbA_L)$ denote the set of adelic points of $X$ and let $\Br(X)$ denote the Brauer group of $X$, $\Br(X)=H^2_{\mathrm{\acute{e}t}}(X,\bbG_m)$. There is a pairing 
$$X(\bbA_L)\times \Br(X)\rightarrow \bbQ/\bbZ$$
obtained by evaluating an element of $\Br(X)$ at an adelic point and summing the local invariants \cite{Manin}. 
The Brauer-Manin set $X(\bbA_L)^{\Br(X)}$ is the set of adelic points of $X$ which are orthogonal to $\Br(X)$ under this pairing. 
It contains the closure of the set of rational points in the adelic topology.
$$\overline{X(L)}\subset X(\bbA_L)^{\Br(X)}\subset X(\bbA_L).$$
If $X(\bbA_L)\neq\emptyset$ but $X(\bbA_L)^{\Br(X)}=\emptyset$, there is said to be a Brauer-Manin obstruction to the Hasse principle on $X$. If $X(\bbA_L)\neq X(\bbA_L)^{\Br(X)}$, there is said to be a Brauer-Manin obstruction to weak approximation on $X$.

Since Manin's observation, Brauer groups and the associated obstructions have been the subject of a great deal of research. Let $\overline{X}$ denote the base change of $X$ to an algebraic closure of $L$. The kernel of the natural map from $\Br(X)$ to $\Br(\overline{X})$ is called the `algebraic' part of $\Br(X)$ and denoted $\Br_1(X)$. 
It is usually easier to handle than the remaining `transcendental' part and a substantial portion of the literature is devoted to its study. 
The quotient group $\Br(X)/\Br_1(X)$, known as the transcendental part of $\Br(X)$, is generally more mysterious. Nevertheless, it has arithmetic importance -- transcendental elements in $\Br(X)$ can obstruct the Hasse principle and weak approximation, as shown by Harari in \cite{Harari} and Wittenberg in \cite{Wittenberg}. 

Results of Skorobogatov and Zarhin in \cite{SZtorsion} allow one to compute the transcendental part of the Brauer group for a product of elliptic curves. These results were used by Ieronymou, Skorobogatov and Zarhin in \cite{ISZ} to give a sufficient condition for the Brauer group of a diagonal quartic surface $D$ over $\bbQ$ to be algebraic, and to give an upper bound on the size of the quotient of $\Br(D)$  by the image of $\Br(\bbQ)$. This built upon earlier work of Ieronyomou \cite{Ieronymou}, who gave a sufficient condition for the $2$-primary part of the Brauer group of a diagonal quartic surface over $\bbQ$ to be algebraic. In \cite{I-S}, Ieronymou and Skorobogatov went on to use the results of \cite{SZtorsion} and \cite{ISZ} to compute the odd order torsion in the quotient of $\Br(D)$  by the image of $\Br(\bbQ)$ for a diagonal quartic surface  $D$ over $\bbQ$.

In this paper, we use the results of \cite{SZtorsion} to compute the transcendental part of the Brauer group for abelian surfaces of the form $E\times E$ where $E/L$ is an elliptic curve with complex multiplication by the ring of integers $\cO_K$ of an imaginary quadratic field $K$.

Skorobogatov and Zarhin \cite{SZ} proved that for $X$ an abelian variety or K3 surface, $\Br(X)/\Br_1(X)$ is a finite abelian group. Therefore, computing $\Br(X)/\Br_1(X)$ is equivalent to computing its $\ell$-primary part $(\Br(X)/\Br_1(X))_{\ell^\infty}$ for every prime number $\ell$. To a pair $(E,\ell)$ consisting of an elliptic curve $E$ defined over a number field $L$, with complex multiplication by $\cO_K$, and a prime number $\ell$, we associate an integer $m(\ell)$ (Definition \ref{def:m}) which can be calculated using class field theory (Proposition \ref{upper bound}). We write $\Gamma_L$ for the absolute Galois group of $L$. We denote the $n$-torsion subgroup of an abelian group $A$ by $A_n$. For an elliptic curve $E/L$, we write $E_n$ for the $n$-torsion points of $E$ defined over an algebraic closure of $L$.

\begin{theorem}
Let $\ell\in\bbZ_{>0}$ be an odd prime and let $m=m(\ell)$. Then
 \begin{eqnarray*}
 \left(\frac{\Br(E\times E)}{\Br_1(E\times E)}\right)_{\ell^{\infty}}= 
\frac{\Br(E\times E)_{\ell^{m}}}{\Br_1(E\times E)_{\ell^{m}}}=\frac{\End_{\Gamma_L} E_{\ell^m}}{(\cO_K\otimes\bbZ/\ell^m)^{\Gamma_L}}
 \cong \begin{cases} (\bbZ/\ell^m)^2 & \textrm{if } K\subset L\\
\bbZ/\ell^m & \textrm{if } K\not\subset L.
\end{cases}
  \end{eqnarray*}
\end{theorem}
For brevity, here we state only the result for odd primes. The results for all primes can be found in Theorems \ref{quotient1} and \ref{quotient2}.
 In Theorems \ref{geom1} and \ref{geom2}, we give a similar description of the $\ell$-primary part of $\Br(\overline{E}\times\overline{E})^{\Gamma_L}$ 
 for every prime $\ell$.
One can apply these results to gain information about the transcendental part of the Brauer group for a wider class of varieties. If $\pi:X\dasharrow Y$ is a dominant rational map of degree $d$ between K3 or abelian surfaces over $L$, then by the proof of \cite{I-S} Corollary 2.2, it induces a surjective map of $\Gamma_L$-modules
$$\pi^*:\Br(\overline{Y})\rightarrow\Br(\overline{X})$$
whose kernel is annihilated by $d$. Thus, if $\ell$ is prime and coprime to $d$, then
\[\left(\frac{\Br(Y)}{\Br_1(Y)}\right)_{\ell^{\infty}}\hookrightarrow
\Br(\overline{Y})^{\Gamma_L}_{\ell^{\infty}}=\Br(\overline{X})^{\Gamma_L}_{\ell^{\infty}}.\]
The following examples are of interest. Suppose that $E/L$ has complex multiplication by $\cO_K$.
\begin{enumerate}
\item $Y=E\times E'$ where $E'/L$ is an elliptic curve which is isogenous to $E$ over $L$. Take $\ell$ coprime to the degree of the isogeny.
\item $Y=E'\times E'$ where $E'/L$ is an elliptic curve with complex multiplication by a non-maximal order $\cO\subset \cO_K$. Take $\ell$ coprime to the index $[\cO_K:\cO]$. This is because there is an isogeny of degree $[\cO_K:\cO]$, defined over $L$, from $E'$ to an elliptic curve over $L$ with complex multiplication by $\cO_K$. 
\item $Y=\Kum(E\times E)$, the K3 surface which is the minimal desingularisation of the quotient of $E\times E$ by the involution $(P,Q)\mapsto (-P,-Q)$. 
\end{enumerate}

\medskip

More is known for a Kummer surface $X=\Kum(E\times E)$. 
By Proposition 1.3 of \cite{SZtorsion}, there is an isomorphism of $\Gamma_L$-modules
\[\Br(\overline{X})\rightarrow \Br(\overline{E}\times\overline{E})\]
and therefore
\[\Br(\overline{X})^{\Gamma_L}=\Br(\overline{E}\times\overline{E})^{\Gamma_L}.\]
By Theorem 2.4 of \cite{SZtorsion}, for every $n\in\bbZ_{>0}$ there is an embedding
\begin{equation}
\label{eq:embedding}
\Br(X)_n/\Br_1(X)_{n}
\hookrightarrow\Br(E\times E)_{n}/\Br_1(E\times E)_{n}
\end{equation}
which is an isomorphism if $n$ is odd. So for $\ell$ an odd prime, 
\begin{equation}
(\Br(X)/\Br_1(X))_{\ell^{\infty}}=(\Br(E\times E)/\Br_1(E\times E))_{\ell^{\infty}}.
\end{equation} 

Examples involving K3 surfaces are important for applications because for abelian varieties with finite Tate-Shafarevich group, any Brauer-Manin obstruction can be explained by the algebraic part of the Brauer group, see \S 6.2 of \cite{Skorobogatov}. However, for K3 surfaces the transcendental part of the Brauer group can play an essential role in the Brauer-Manin obstruction.
Examples of this are given in \cite{Ieronymou}, \cite{HVV}, \cite{Preu} and \cite{I-S}. We give another example in Theorem \ref{thm:Brauer-Maninset} below. We focus on elliptic curves with a transcendental element of odd order in $\Br(E\times E)$ because this will give rise to a transcendental element in the Brauer group of $\Kum(E\times E)$.

\begin{theorem}
\label{thm:quadratictwist}
Let $E/\bbQ$ be an elliptic curve with complex multiplication by $\cO_K$ such that $\Br(E\times E)$ contains a transcendental element of odd order.
 Then $E$ has affine equation $y^2=x^3+2c^3$ for some $c\in\bbQ^\times$. Moreover, for $X=\Kum(E\times E)$ we have $\Br_1(X)=\Br(\bbQ)$ and \[\Br(X)/\Br(\bbQ)=\Br(X)_3/\Br(\bbQ)_3=\Br(E\times E)_3/\Br_1(E\times E)_3\cong \bbZ/3.\]
\end{theorem}
For $c\in\bbQ^\times$, let $E^c$ denote the elliptic curve over $\bbQ$ with affine equation $y^2=x^3+2c^3$. Let $X=\Kum(E^c\times E^c)$ denote the Kummer surface, which is independent of the choice of $c\in\bbQ^\times$.
One consequence of Theorem \ref{thm:quadratictwist} is that the algebraic elements of $\Br(X)$ do not produce a Brauer-Manin obstruction. The following theorem shows that a transcendental Brauer element gives rise to a Brauer-Manin obstruction to weak approximation on $X$. 

\begin{theorem}
\label{thm:Brauer-Maninset}
Let \mbox{$\cA\in\Br(X)_3\setminus \Br(\bbQ)$}. Let $\nu$ be a place of $\bbQ$. Then the evaluation map 
\[\ev_{\cA, \nu}:X(\bbQ_\nu)\rightarrow \Br(\bbQ_v)_3\]
is surjective for $\nu=3$ and zero for every other place. Consequently,  
\[X(\bbA_{\bbQ})^{\Br(X)}=X(\bbQ_3)_0\times X(\bbR)\times \prod_{\ell\neq 3}{X(\bbQ_{\ell})}\ \subsetneq\  X(\bbA_{\bbQ})\]
where $X(\bbQ_3)_0$ denotes the points $P\in X(\bbQ_3)$ with $\ev_{\cA,3}(P)=0$, and the product runs over prime numbers $\ell\neq 3$.
\end{theorem}

The structure of the paper is as follows. Section \ref{compute} is devoted to the computation of the transcendental part of the Brauer group of $E\times E$ for a CM elliptic curve $E$. Section \ref{examples} contains applications of these results to special cases and explicit examples. The proof of Theorem \ref{thm:quadratictwist} is carried out in Section \ref{Kummer}. In Section \ref{obstruction}, we compute the Brauer-Manin obstruction to weak approximation on $\Kum(E\times E)$ for $E/\bbQ$ (a quadratic twist of) the elliptic curve with affine equation \mbox{$y^2=x^3+2$}, leading to a proof of Theorem \ref{thm:Brauer-Maninset}.

\subsection{Notation}
We fix the following notation.

\smallskip

\begin{tabular}{ll}

$K$ & an imaginary quadratic field\\
$\cO_K$ & the ring of integers of $K$\\
$\Delta_K$ & the discriminant of $K$\\
$H_K$ & the Hilbert class field of $K$\\
$h(\cO_K) $ & the class number of $\cO_K$, $h(\cO_K)=[H_K:K]$\\
\end{tabular}

\begin{tabular}{ll}
$L$ &  a number field\\
$\overline{L}$ & an algebraic closure of $L$ such that $H_K\subset \overline{L}$\\ 
$\Gamma_F$ & the absolute Galois group of  a field $F$\\

$\mu_n$ & the group of $n$th roots of unity\\
$\zeta_n$  & a primitive $n$th root of unity\\
$E$ & an elliptic curve over $L$ with complex multiplication by $\cO_K$\\
$\overline{E}$ & the base change of $E$ to $\overline{L}$, $\overline{E}=E\times_L \overline{L}$\\
$E_n$ & the $n$-torsion points of $E$ defined over $\overline{L}$\\
$E_n(F)$ & the $n$-torsion points of $E$ defined over a field extension $F$ of $L$\\
$\Kum(E\times E)$ & the K3 surface which is the minimal desingularisation \\
& of the quotient of $E\times E$ by the involution $(P,Q)\mapsto (-P,-Q)$\\
$f_{\frq/\frp}$ & the residue class degree $f_{\frq/\frp}=[\cO_M/\frq:\cO_F/\frp]$ for a prime $\frq$ in \\& a number field $M$ lying above a prime $\frp$ in a subfield $F\subset M$.\\
\end{tabular}\\
\smallskip

 For any $c\in\bbZ_{>0}$, we use the following notation.  \\

\begin{tabular}{ll}
$\cO_{c}$ & the order $\bbZ+c\cO_K$ of conductor $c$ in $\cO_K$\\
$K_c$ & the ring class field corresponding to the order $\cO_c$. \\
\end{tabular}\\

Specifically, let $I_K(c)$ denote the group of fractional $\cO_K$-ideals coprime to $c$, and let $P_{K,\bbZ}(c)$ be the subgroup generated by the principal ideals $\alpha\cO_K$ where $\alpha\in\cO_c$ and $\alpha$ is coprime to $c$.
Then $K_c$ is defined to be the extension of $K$ that corresponds under class field theory to $P_{K,\bbZ}(c)$. The Artin map gives an isomorphism $I_K(c)/P_{K,\bbZ}(c)\cong \Gal(K_c/K)$. In particular, we will make use of the ring class fields $K_{\ell^n}$ for a prime number $\ell$ and $n\in\bbZ_{\geq 0}$. They enjoy the following properties which will be needed later.

\begin{enumerate}
\item $K_{\ell^0}=H_K$,
\item $K_{\ell^n}\subset K_{\ell^{n+1}}$ for $n\geq 0$,
\item the degree $[K_{\ell^n}:K]$ tends to infinity as $n$ tends to infinity. 
\end{enumerate}
For more details on ring class fields, see \cite{Cox} for example.
\medskip

For an abelian group $A$ and an integer $n\in\bbZ_{>0}$, we write $A_n$ for the elements of order dividing $n$ in $A$. For a prime number $\ell\in\bbZ_{>0}$, we write $A_{\ell^{\infty}}$ for the $\ell$-primary part of the abelian group $A$.

\medskip

For $x\in\bbR$, let $\floor{x}$, $\ceil{x}$ denote the floor and ceiling of $x$ respectively.


\section{Transcendental Brauer group computations}
\label{compute}
\subsection{Preliminaries}

Let $L$ be a number field and let $\Gamma_L$ denote its absolute Galois group. In \cite{SZtorsion}, for $A=E\times E'$ a product of elliptic curves defined over $L$ and for every $n\in\bbZ_{>0}$, Skorobogatov and Zarhin gave a canonical isomorphism of \mbox{$\Gamma_L$-modules} 
\begin{equation}
\label{n-torsion}
\Br(\overline{A})_n=\Hom(E_n, E'_n)/(\Hom(\overline{E},\overline{E'})\otimes\bbZ/n)
\end{equation}
and a canonical isomorphism of abelian groups
\begin{equation}
\label{quotient}
\Br(A)_n/\Br_1(A)_n=\Hom_{\Gamma_L}(E_n,E'_n)/(\Hom(\overline{E},\overline{E'})\otimes\bbZ/n)^{\Gamma_L}.
\end{equation}
They used this concrete description of the transcendental part of the Brauer group to give many examples for which $\Br(A)/\Br_1(A)$ is trivial or a finite abelian $2$-group.

\smallskip

From now on, we fix an elliptic curve $E/L$ with complex multiplication by $\cO_K$. 
We begin with a simple observation which enables us to use \eqref{quotient} to compute \mbox{$(\Br(E\times E)/\Br_1(E\times E))_{\ell^{\infty}}$}.
\begin{lemma}
Let $X$ be a smooth, projective, geometrically irreducible variety over a number field. Then for any prime number $\ell$, we have 
\[(\Br(X)/\Br_1(X))_{\ell^{\infty}}=\Br(X)_{\ell^{\infty}}/\Br_1(X)_{\ell^{\infty}}.\]
\end{lemma}
\begin{proof}
Since $X$ is smooth, Proposition 1.4 of \cite{Grothendieck} tells us that $\Br(X)$ is a torsion abelian group.
It follows that the natural inclusion
\[\Br(X)_{\ell^{\infty}}/\Br_1(X)_{\ell^{\infty}}\hookrightarrow (\Br(X)/\Br_1(X))_{\ell^{\infty}}\]
is an equality. 
\end{proof}

To each prime number $\ell\in\bbZ_{>0}$ we associate an integer $m(\ell)$ which will appear in our description of the $\ell$-primary part of the transcendental Brauer group of $E\times E$. In order to define $m(\ell)$, we use  the Gr\"{o}ssencharacter $\psi_{E/KL}$ of $E$ considered as an elliptic curve over $KL$. Recall that $\psi_{E/KL}$ is unramified at the primes of $KL$ of good reduction for $E$. Therefore, for such primes we write $\psi_{E/KL}(\frq)$ for the evaluation of $\psi_{E/KL}$ at an idele $(\dots , 1,1,\pi_{\frq},1,1,\dots)\in \bbA^{\times}_{KL}$ where the entry $\pi_{\frq}$ at the prime $\frq$ is a uniformiser at $\frq$.

\begin{definition}
\label{def:m}
For a prime number $\ell\in\bbZ_{>0}$, let $m(\ell)$ be the largest integer $k$ such that for all 
primes $\frq$ of $KL$ which are of good reduction for $E$ and coprime to $\ell$,  
the Gr\"{o}ssencharacter 
$\psi_{E/KL}$ satisfies $$\psi_{E/KL}(\frq)\in\cO_{\ell^k}=\bbZ+\ell^k\cO_K.$$
\end{definition}

We define an auxiliary integer $n(\ell)$ which aids computation of $m(\ell)$ and in most cases removes the dependence on the Gr\"{o}ssencharacter.

\begin{definition}
For a prime number $\ell\in\bbZ_{>0}$, let $n(\ell)$ be the largest integer $k$ for which the ring class field $K_{\ell^k}$ of the order $\cO_{\ell^k}$ embeds into $KL$.
\end{definition}

\begin{proposition}
\label{upper bound}
Let $\ell\in\bbZ_{>0}$ be prime. Then $$m(\ell)\leq n(\ell)$$ with equality if $\cO_K^*=\{\pm 1\}$ (in other words, if $K\notin\{\bbQ(i),\bbQ(\zeta_3)\}$).
\end{proposition}

\begin{proof}
Write $m=m(\ell)$ and $n=n(\ell)$. Let $S$ be the set of primes $\frr$ of $KL$ that satisfy at least one of the following five conditions: \begin{enumerate} 
\item $\frr\mid\infty$,  
\item $\frr\mid \ell$,
\item $E$ has bad reduction at $\frr$, 
\item $\frr$ is ramified in $K_{\ell^{n+1}}L/K$, 
\item $\psi_{E/KL}(\frr)\notin\cO_{\ell^{n+1}}$.\end{enumerate}
Suppose for contradiction that \mbox{$m\geq n+1$}, and hence $S$ is a finite set. Then, since $K_{\ell^{n+1}}\nsubseteq KL$, Exercise 6.1 of \cite{CF} tells us that there exists a prime $\frq$ of $KL$ with $\frq\notin S$ which does not split completely in $K_{\ell^{n+1}}L/KL$. Let $\frp=\frq\cap\cO_K$. Let $f_{\frq/\frp}$ denote the residue class degree of $\frq$ over $\frp$, $f_{\frq/\frp}=[\cO_{KL}/\frq:\cO_K/\frp]$. By Theorems II.9.1 and II.9.2 of \cite{Silverman}, the Gr\"{o}ssencharacter $\psi_{E/KL}$ sends $\frq$ to a generator of the principal ideal $N_{KL/K}(\frq)=\frp^{f_{\frq/\frp}}$. Consider the following diagram of field extensions.
\[
\xymatrix{&& K_{\ell^{n+1}}L\ar@{-}[dr] \ar@{-}[dl]& & \\
&K_{\ell^{n+1}}\ar@{-}[dr] & & KL\ar@{-}[dl] &\frq \\
& &K &  \frp &\\
}
\]
Recall that, given an abelian extension of number fields $M/N$ and a prime ideal $\frr$ of $\cO_N$ that is unramified in $M/N$, the Artin symbol $(\frr,M/N)$ is the unique element $\sigma\in\Gal(M/N)$ such that for all $\alpha\in \cO_M$, $$\sigma(\alpha)\equiv \alpha^{N_{F/\bbQ}(\frr)}\pmod{\frs}$$
where $\frs$ is a prime of $M$ above $\frr$. In our case, $N_{KL/\bbQ}(\frq)=N_{K/\bbQ}(N_{KL/K}(\frq))=N_{K/\bbQ}(\frp)^{f_{\frq/\frp}}$. Therefore, the restriction of the Artin symbol $(\frq,K_{\ell^{n+1}}L/KL)$ to $K_{\ell^{n+1}}/K$ (in other words, its image in $\Gal(K_{\ell^{n+1}}/K)$) satisfies
\begin{eqnarray*}
\Res_{K_{\ell^{n+1}}/K}(\frq,K_{\ell^{n+1}}L/KL)=(\frp,K_{\ell^{n+1}}/K)^{f_{\frq/\frp}}
&=&(\frp^{f_{\frq/\frp}},K_{\ell^{n+1}}/K)\\
&=&((\psi_{E/KL}(\frq)),K_{\ell^{n+1}}/K).
\end{eqnarray*}
Since $\frq\notin S$, we have $\psi_{E/KL}(\frq)\in\cO_{\ell^{n+1}}$ and hence \[((\psi_{E/KL}(\frq)),K_{\ell^{n+1}}/K)=1\]
by definition of the ring class field $K_{\ell^{n+1}}$.
But this implies that 
\[\Res_{K_{\ell^{n+1}}/K}(\frq,K_{\ell^{n+1}}L/KL)=1\]
and therefore
\[(\frq,K_{\ell^{n+1}}L/KL)=1.\]
This is a contradiction because $\frq$ does not split completely in $K_{\ell^{n+1}}L/KL$.
Therefore, $m\leq n$. It remains to show that $m=n$ when $\cO_K^*=\{\pm 1\}$. From now on, suppose that $\cO_K^*=\{\pm 1\}$. Let $\frq$ be a finite prime of $KL$ of good reduction for $E$ which is coprime to $\ell$ and unramified in $KL/K$. Let $\frp=\frq\cap\cO_K$ and let $\frs=\frq\cap\cO_{K_{\ell^n}}$. The Artin symbol $(\frp,K_{\ell^n}/K)$ has order $f_{\frs/\frp}$ in $\Gal(K_{\ell^n}/K)$. Since $K\subset K_{\ell^n}\subset KL$, we have $f_{\frs/\frp}\mid f_{\frq/\frp}$, whereby 
\[1=(\frp,K_{\ell^n}/K)^{f_{\frq/\frp}}=
(\frp^{f_{\frq/\frp}}, K_{\ell^n}/K)=(N_{KL/K}(\frq),K_{\ell^n}/K).\]
By definition of the ring class field $K_{\ell^n}$, this implies that $$N_{KL/K}(\frq)=\alpha\cO_K$$ for some $\alpha\in\cO_{\ell^n}$. But $\psi_{E/KL}(\frq)$ is a generator of $N_{KL/K}(\frq)$ and $\cO_K^*=\{\pm 1\}$ so this implies that $\psi_{E/KL}(\frq)\in\cO_{\ell^n}$, as required.
\end{proof}

\begin{remark}
\label{rem:ringclassdegree}
Class field theory gives $[K_c:K]=h(\cO_c)$, where $h(\cO_c)$ denotes the class number of the order $\cO_c$.
The following formula for $h(\cO_c)$ can be found in \cite{Cox}, Theorem 7.24, for example. 
\begin{equation}
\label{degreeformula}
[K_c:K]=h(\cO_c)=\frac{h(O_K) c}{ [\cO_K^*:\cO_c^*]}\prod_{p\mid c}{\Bigl(1-\Bigl(\frac{\Delta_K}{p}\Bigr)\frac{1}{p}\Bigr)}
\end{equation}
where the product is taken over the prime factors of $c$. The symbol $(\frac{\Delta_K}{p})$ denotes the Legendre symbol for odd primes. For the prime $2$, the Legendre symbol is replaced by the Kronecker symbol $(\frac{\Delta_K}{2})$, defined as
$$\Bigl(\frac{\Delta_K}{2}\Bigr)=\begin{cases}
0 & \textrm{if } 2\mid \Delta_K \\
1 & \textrm{if } \Delta_K\equiv 1\pmod{8}\\
-1 & \textrm{if } \Delta_K\equiv 5\pmod{8}.
\end{cases}$$

If $K_{\ell^k}\subset KL$, then $[K_{\ell^k}:K]$ divides $[KL:K]$. Thus, in any given example, \eqref{degreeformula} allows one to identify a finite set of primes $S$ such that $m(\ell)=n(\ell)=0$ for all $\ell\notin S$. For a prime $\ell$ in $S$, \eqref{degreeformula}  gives an upper bound for $n(\ell)$, and therefore also an upper bound for $m(\ell)$. For $K\in\{\bbQ(i), \bbQ(\zeta_3)\}$, one must examine the Gr\"{o}ssencharacter in order to compute $m(\ell)$. For explicit descriptions of Gr\"{o}ssencharacters for elliptic curves with complex multiplication by $\bbQ(i)$ or $\bbQ(\zeta_3)$, see \cite{R-S} Theorems 5.6 and 5.7 respectively.

\end{remark}

In Sections \ref{case1} and \ref{case2}, we will use the isomorphisms \eqref{n-torsion} and \eqref{quotient} 
to compute the $\ell$-primary part of the transcendental Brauer group of $E\times E$ in terms of endomorphisms of the $\ell$-power torsion of $E$. For this, we will need two auxiliary lemmas. Before stating the first lemma, we note that since $\End \overline{E}=\cO_K$, there is a natural map $\cO_K\otimes\bbZ/\ell^k\rightarrow \End E_{\ell^k}$.  The injectivity of this map follows from the fact that the $\ell$-adic Tate module $T_\ell(E)$ is a free $\cO_K\otimes \mathbb{Z}_\ell$-module of rank $1$. (For a more general fact about abelian varieties, see Proposition 2.2.1 of \cite{Ribet}.) Thus, we may view $\cO_K\otimes\bbZ/\ell^k$ as a subring of $\End E_{\ell^k}$.

\begin{lemma}
\label{E+}
Let $\ell\in\bbZ_{>0}$ be prime, let $k\in\bbZ_{\geq 0}$ and let $$(\End E_{\ell^k})^+=\{\psi\in\End E_{\ell^k}\mid \psi x=x\psi \ \forall x \in \cO_K\}.$$ Then, viewing $\cO_K\otimes\bbZ/\ell^k$ as a subring of $\End E_{\ell^k}$, we have
$$(\End E_{\ell^k})^+=\cO_K\otimes\bbZ/\ell^k.$$
\end{lemma}

\begin{proof}
As an abelian group, $E_{\ell^k}\cong (\bbZ/\ell^k)^2$, and therefore $\End E_{\ell^k}\cong M_2(\bbZ/\ell^k)$.
The proof comes down to an easy calculation with two-by-two matrices with entries in $\bbZ/\ell^k$.
\end{proof}

\begin{lemma}
\label{fixed}
Let $\ell\in\bbZ_{>0}$ be prime and let $m=m(\ell)$. Let $k\in\bbZ_{\geq 0}$ and let $\vphi\in\End E_{\ell^k}$. Then 
\begin{enumerate}
\item \label{partclassfixed}The class of $\varphi$ in $\End E_{\ell^{k}} /(\cO_K\otimes\bbZ/\ell^{k})$ is fixed by $\Gamma_{KL}$ if and only if for all $x\in\cO_K$, 
$$\ell^m(x\vphi-\vphi x)\in(\End E_{\ell^k})^+=\cO_K\otimes\bbZ/\ell^k.$$
\item \label{partendofixed}The endomorphism $\vphi$ is fixed by $\Gamma_{KL}$ if and only if
$$\ell^m\vphi\in(\End E_{\ell^k})^+=\cO_K\otimes\bbZ/\ell^k.$$

\end{enumerate}
\end{lemma}

\begin{proof}
The action of $\Gamma_{KL}$ on $\End E_{\ell^{k}}$ factors through the abelian Galois group $\Gal(KL(E_{\ell^{k}})/KL)$. Let $\frq$ be a finite prime of $KL$ which is coprime to $\ell$ and of good reduction for $E$. The N\'{e}ron-Ogg-Shafarevich criterion tells us that $\frq$ is unramified in $KL(E_{\ell^{k}})/KL$. Since $E$ has complex multiplication by $\cO_K$, the Artin symbol $(\frq, KL(E_{\ell^{k}})/KL)$ acts on $E_{\ell^k}$ as multiplication by $\psi_{E/KL}(\frq)$. For a proof of this fact, see \cite{Lang}, Ch. 4, Corollary 1.3 (iii), for example. Therefore, the action of $(\frq, KL(E_{\ell^{k}})/KL)$ on $\End(E_{\ell^k})$ is conjugation by $\psi_{E/KL}(\frq)$. 
The Artin symbols for the unramified primes generate $\Gal(KL(E_{\ell^{k}})/KL)$.

Let $\alpha=(\Delta_K+\sqrt{\Delta_K})/2$, so $\cO_K=\bbZ[\alpha]$. Let $a, b\in\bbZ$ be such that $a+b\alpha$ is invertible in $\cO_K\otimes\bbZ/\ell^k$. Let $\vphi\in\End E_{\ell^k}$. We have
\[(a+b\alpha)\vphi-\vphi(a+b\alpha)=b(\alpha\vphi-\vphi\alpha).\]
Hence, the class of $\vphi$ in $\End E_{\ell^{k}} /(\cO_K\otimes\bbZ/\ell^{k})$ is fixed by conjugation by $a+b\alpha$ if and only if
\begin{equation}
\label{class fixed}
b(\alpha\vphi-\vphi\alpha)\in\cO_K\otimes\bbZ/\ell^k
\end{equation}
and $\vphi$ is fixed by conjugation by  $a+b\alpha$ if and only if
\begin{equation}
\label{varphi fixed}
b(\alpha\vphi-\vphi\alpha)=0.
\end{equation}

 Recall that $m=m(\ell)$ is the largest integer $t$ such that for all finite primes $\frq$ of $KL$ which are of good reduction for $E$ and coprime to $\ell$, 
$$\psi_{E/KL}(\frq)\in\cO_{\ell^t}=\bbZ+\ell^t\cO_K.$$
In other words, for a prime $\frq$ which is unramified in $KL(E_{\ell^k})/KL$, we can write \mbox{$\psi_{E/KL}(\frq)=a+b\alpha$} for some $a,b\in\bbZ$ with $\ord_{\ell}(b)=m$.
Hence, by \eqref{class fixed}, the class of $\vphi$ in $\End E_{\ell^{k}} /(\cO_K\otimes\bbZ/\ell^{k})$ is fixed by $\Gamma_{KL}$ if and only if
\begin{eqnarray*}
 \ell^m(\alpha\vphi-\vphi\alpha)\in \cO_K\otimes\bbZ/\ell^k.
\end{eqnarray*}
By \eqref{varphi fixed}, the endomorphism $\vphi$ is fixed by $\Gamma_{KL}$ if and only if
\begin{eqnarray*}
 \ell^m(\alpha\vphi-\vphi\alpha)=0.
\end{eqnarray*}
An application of Lemma \ref{E+} completes the proof.
\end{proof}

\subsection{Case I: Complex multiplication defined over the base field.}
\label{case1}

In this subsection, we compute the transcendental Brauer group of $E\times E$ in the case where the complex multiplication field $K$ is a subfield of $L$, the field of definition of $E$.
\begin{theorem}
\label{quotient1}
Suppose that $K\subseteq L$.
Let $\ell\in\bbZ_{>0}$ be prime and let $m=m(\ell)$. Then
 \begin{eqnarray*}
 \left(\frac{\Br(E\times E)}{\Br_1(E\times E)}\right)_{\ell^{\infty}}= 
\frac{\Br(E\times E)_{\ell^{m}}}{\Br_1(E\times E)_{\ell^{m}}}=\frac{\End E_{\ell^m}}{\cO_K\otimes\bbZ/\ell^m}
 \cong  (\bbZ/\ell^m)^2.
  \end{eqnarray*}
\end{theorem}

\begin{proof}
By \eqref{quotient}, for all primes $\ell$ and all $k\in\bbZ_{\geq 0}$, we have 
\[\frac{\Br(E\times E)_{\ell^{k}}}{\Br_1(E\times E)_{\ell^{k}}}=\frac{\End _{\Gamma_L}E_{\ell^k}}{\cO_K\otimes\bbZ/\ell^k}.\]
Also, 
\[\frac{\End E_{\ell^k}}{\cO_K\otimes\bbZ/\ell^k}\cong (\bbZ/\ell^k)^2.\]
The result now follows from Lemma \ref{fixed}, part \ref{partendofixed}.
\end{proof}

\begin{theorem}
\label{geom1}
Suppose that $K\subseteq L$.
 Let $\ell\in\bbZ_{>0}$ be prime and let $m=m(\ell)$. Then
 \begin{eqnarray*}
 \Br(\overline{E}\times\overline{E})_{\ell^{\infty}}^{\Gamma_{L}}&=&
\left(\frac{\End E_{\ell^{m+\ceil{\ord_{\ell}(\Delta_K)/2}}}}{\cO_K\otimes\bbZ/\ell^{m+\ceil{\ord_{\ell}(\Delta_K)/2}}}\right)^{\Gamma_L}\\
&\cong & \bbZ/\ell^{m+\floor{\ord_{\ell}(\Delta_K)/2}}\times \bbZ/\ell^{m+\ceil{\ord_{\ell}(\Delta_K)/2}}.
 \end{eqnarray*}
In particular, if $\ell\nmid \Delta_K$ then
\[\Br(\overline{E}\times\overline{E})_{\ell^{\infty}}^{\Gamma_{L}}
=\frac{\End E_{\ell^m}}{\cO_K\otimes\bbZ/\ell^m}\cong (\bbZ/\ell^m)^2.\]
\end{theorem}

\begin{proof}
Fix a prime number $\ell\in\bbZ_{>0}$ and let $k\in\bbZ_{\geq 0}$. By \eqref{n-torsion}, we have
\[\Br(\overline{E}\times\overline{E})_{\ell^{k}}^{\Gamma_{L}}=\left(\frac{\End E_{\ell^k} }{\cO_K\otimes\bbZ/\ell^k}\right)^{\Gamma_L}.  \]
Write $\cO_K=\bbZ[\alpha]$ where $\alpha=(\Delta_K+\sqrt{\Delta_K})/2$ and let $\vphi\in\End E_{\ell^k}$. By part \ref{partclassfixed} of Lemma \ref{fixed}, the class of $\vphi$ in $\End E_{\ell^{k}} /(\cO_K\otimes\bbZ/\ell^{k})$ is fixed by $\Gamma_{L}$ if and only if 
\begin{equation}
\label{classfixed}
\ell^m(\alpha\vphi-\vphi \alpha)\in\cO_K\otimes\bbZ/\ell^k.
\end{equation}
Let $P, \alpha P$ be a $\bbZ/\ell^{k}$-basis for $E_{\ell^{k}}$. With respect to this basis, multiplication by $\alpha$ is given by the following matrix:
$$\begin{pmatrix}0& \ \ \ \ \frac{\Delta_K(1-\Delta_K)}{4}\\ 1 & \Delta_K\end{pmatrix}.$$
Subtracting an element of $\cO_K\otimes\bbZ/\ell^k$ if necessary, we may assume that $\vphi$ is of the form
$$\begin{pmatrix}0&\ \ t\\ 0 &\ \  u\end{pmatrix}$$
for some $t,u\in\bbZ/\ell^k$. In terms of matrices, equation \eqref{classfixed} becomes
\[\begin{pmatrix} -\ell^mt &\ \ \  -\ell^mt\Delta_K+\ell^m u\frac{\Delta_K(1-\Delta_K)}{4} \\-\ell^mu& \ell^m t\end{pmatrix}=\begin{pmatrix}a& \ \ \ b\frac{\Delta_K(1-\Delta_K)}{4}\\ b & a+b\Delta_K\end{pmatrix}\]
for some $a, b\in\bbZ/\ell^k$. The resulting equations reduce to
\begin{equation}
\label{uveqns}
2\ell^mt\equiv\ell^m\Delta_Kt\equiv\ell^m\Delta_K u\equiv\ell^m\frac{\Delta_K(1-\Delta_K)}{2}u\equiv 0\pmod {\ell^k}.
\end{equation}
We have $\ord_2(\Delta_K)\in\{0,2 ,3\}$ and for an odd prime $\ell$, $\ord_\ell(\Delta_K)\in\{0, 1\}$. Thus, \eqref{uveqns} can be summarised as
\[\ell^{m+\floor{\ord_{\ell}(\Delta_K)/2}}t\equiv \ell^{m+\ceil{\ord_{\ell}(\Delta_K)/2}}u\equiv 0\pmod{\ell^k}. \]
Therefore, 
\begin{eqnarray*}
\Br(\overline{E}\times\overline{E})_{\ell^{\infty}}^{\Gamma_{L}}&=&
 \Br(\overline{E}\times\overline{E})_{\ell^{m+\ceil{\ord_{\ell}(\Delta_K)/2}}}^{\Gamma_{L}}\\
&=&\left(\frac{\End E_{\ell^{m+\ceil{\ord_{\ell}(\Delta_K)/2}}}}{\cO_K\otimes\bbZ/\ell^{m+\ceil{\ord_{\ell}(\Delta_K)/2}}}\right)^{\Gamma_L}\\
&\cong & \bbZ/\ell^{m+\floor{\ord_{\ell}(\Delta_K)/2}}\times \bbZ/\ell^{m+\ceil{\ord_{\ell}(\Delta_K)/2}}.
 \end{eqnarray*}
\end{proof}

\begin{remark}
The fact that $(\Br(E\times E)/\Br_1(E\times E))_{\ell^\infty}=\Br(\overline{E}\times \overline{E})^{\Gamma_L}_{\ell^\infty}$ for $\ell\nmid \Delta_K$ also follows from Proposition 5.2 of \cite{C-T--S}. A computation of the relevant intersection pairing shows that the cokernel of the map $\Br(E\times E)/\Br_1(E\times E)\hookrightarrow \Br(\overline{E}\times \overline{E})^{\Gamma_L}$ is annihilated by the discriminant of $K$.
\end{remark}

\subsection{Case II: Complex multiplication not defined over the base field.}
\label{case2}

Throughout this subsection, we make the assumption that $K\not\subset L$. We write $\tau$ for an element of $\Gamma_L\setminus \Gamma_{KL}.$ 
We set $\alpha=(\Delta_K+\sqrt{\Delta_K})/2$, so $\cO_K=\bbZ[\alpha]$.

\begin{lemma}
\label{tau fixed}
Suppose that $K\nsubseteq L$.
Let $\ell\in\bbZ_{>0}$ be prime and let $k\in\bbZ_{\geq 0}$.  Let $a,b\in\bbZ$ and consider $(a+b\alpha)\tau$ as an element of $\End E_{\ell^k}$. Then 
\begin{enumerate}
\item \label{partclasstauKL}The class of $(a+b\alpha)\tau$ in $\End E_{\ell^k}/(\cO_K\otimes\bbZ/\ell^k)$ is fixed by $\Gamma_{KL}$ if and only if $$\ord_{\ell}(a),\ \ord_{\ell}(b)\geq k-m(\ell)-\ord_{\ell}(\Delta_K).$$
\item \label{partclasstautau}The class of $(a+b\alpha)\tau$ in $\End E_{\ell^k}/(\cO_K\otimes\bbZ/\ell^k)$ is fixed by $\tau$ if and only if
\[
\ord_{\ell}(b)\geq k-\ord_{\ell}(\Delta_K).\]
\item \label{partplus}We have $(a+b\alpha)\tau\in(\End E_{\ell^k})^+=\cO_K\otimes\bbZ/\ell^k$ if and only if
\begin{eqnarray*}
\ord_{\ell}(a)\geq k-\floor{\ord_{\ell}(\Delta_K)/2}\\
\textrm{and }\ 
\ord_{\ell}(b)\geq k-\ceil{\ord_{\ell}(\Delta_K)/2}.
\end{eqnarray*}
\item \label{parttaufixKL}We have $(a+b\alpha)\tau\in\End_{\Gamma_{KL}} E_{\ell^k}$ if and only if 
\begin{eqnarray*}
\ord_{\ell}(a)\geq k-m(\ell)-\floor{\ord_{\ell}(\Delta_K)/2}\\
\textrm{and }\ 
\ord_{\ell}(b)\geq k-m(\ell)-\ceil{\ord_{\ell}(\Delta_K)/2}.
\end{eqnarray*}
\item \label{parttaufixtau}The endomorphism $(a+b\alpha)\tau$ is fixed by the action of $\tau$ if and only if 
\[\ord_{\ell}(b)\geq k-\floor{\ord_{\ell}(\Delta_K)/2}.\]
\end{enumerate}
 
\end{lemma}

\begin{proof}
Write $m=m(\ell)$. 
\begin{enumerate}
\item By part \ref{partclassfixed} of Lemma \ref{fixed}, the class of $(a+b\alpha)\tau$ in \mbox{$\End E_{\ell^k}/(\cO_K\otimes\bbZ/\ell^k)$} is fixed by $\Gamma_{KL}$ if and only if 
\begin{equation}
\label{classfix}
\ell^m(a+b\alpha)(\alpha\tau-\tau\alpha)=\ell^m\sqrt{\Delta_K}(a+b\alpha)\tau\in(\End E_{\ell^k})^+.
\end{equation}
By the definition of $(\End E_{\ell^k})^+$, \eqref{classfix} shows that the class of $(a+b\alpha)\tau$ in $\End E_{\ell^k}/(\cO_K\otimes\bbZ/\ell^k)$ is fixed by $\Gamma_{KL}$ if and only if 
\[\ell^m\sqrt{\Delta_K}(a+b\alpha)(\alpha\tau-\tau\alpha)=\ell^m\Delta_K(a+b\alpha)\tau\equiv 0\pmod{\ell^k}.\]

\item 
The class of $(a+b\alpha)\tau$ in $\End E_{\ell^k}/(\cO_K\otimes\bbZ/\ell^k)$ is fixed by $\tau$ if and only if 
\begin{equation}
\label{fixedbytau}
(a+b\alpha)\tau-\tau(a+b\alpha)\tau\tau^{-1}=b\sqrt{\Delta_K}\tau\in \cO_K\otimes\bbZ/\ell^k.
\end{equation}
By Lemma \ref{E+}, $\cO_K\otimes\bbZ/\ell^k=(\End E_{\ell^k})^+$. So, by \eqref{fixedbytau} and the definition of $(\End E_{\ell^k})^+$, the class of $(a+b\alpha)\tau$ in $\End E_{\ell^k}/(\cO_K\otimes\bbZ/\ell^k)$ is fixed by $\tau$ if and only if 
\[\alpha b\sqrt{\Delta_K}\tau-b\sqrt{\Delta_K}\tau\alpha=b\Delta_K\tau\equiv 0\pmod{\ell^k}.\]

\item
By definition of $(\End E_{\ell^k})^+$, we have
\begin{eqnarray*}
(a+b\alpha)\tau\in(\End E_{\ell^k})^+ \iff
 (a+b\alpha)(\alpha\tau-\tau\alpha)
\equiv 0\pmod{\ell^k}.
\end{eqnarray*}
Expanding $(a+b\alpha)(\alpha\tau-\tau\alpha)$ gives
\[(a+b\alpha)(\alpha\tau-\tau\alpha)=\Bigl(b\frac{\Delta_K(1-\Delta_K)}{2}-\Delta_K a+(2a+b\Delta_K)\alpha\Bigr)\tau.\]
The conditions of part \ref{partplus} are precisely those arising from
\[b\frac{\Delta_K(1-\Delta_K)}{2}-\Delta_K a\equiv 2a+b\Delta_K\equiv 0\pmod{\ell^k}. \]

\item
By part \ref{partendofixed} of Lemma \ref{fixed}, 
\[(a+b\alpha)\tau\in\End_{\Gamma_{KL}} E_{\ell^k}\iff \ell^m(a+b\alpha)\tau\in(\End E_{\ell^k})^+.\]
Now apply part \ref{partplus} of Lemma \ref{tau fixed}.

\item
The endomorphism $(a+b\alpha)\tau$ is fixed by the action of $\tau$ if and only if
\begin{equation}
(a+b\alpha)\tau-\tau(a+b\alpha)\tau\tau^{-1}=b\sqrt{\Delta_K}\tau\equiv0\pmod{\ell^k}.
\end{equation}
It is easily seen that $b\sqrt{\Delta_K}\equiv 0\pmod{\ell^k}$ if and only if 
\[\ord_{\ell}(b)\geq k-\floor{\ord_{\ell}(\Delta_K)/2}.\]
\end{enumerate}
\end{proof}

\begin{theorem}
\label{geom2}
Suppose that $K\nsubseteq L$ and let $\ell\in\bbZ_{>0}$ be prime. Let $m=m(\ell)$ and let $k=m+\ord_{\ell}(\Delta_K)$. Let $\theta$ denote the image of $\tau$ in the quotient group $\End E_{\ell^{k}}/(\cO_K\otimes\bbZ/\ell^{k})$. Then $$\Br(\overline{E}\times\overline{E})_{\ell^{\infty}}^{\Gamma_{KL}}=\cO_{K}\theta$$ and
\begin{eqnarray*}
 \Br(\overline{E}\times\overline{E})_{\ell^{\infty}}^{\Gamma_{L}}&=&\cO_{\ell^m}\theta\cong\begin{cases}
\bbZ/\ell^{k} & \textrm{if $\ell$ is odd or $\ell\nmid\Delta_K$}\\
\bbZ/2^{k-1}\times\bbZ/2 & \textrm{if $\ell=2$ and $2\mid\Delta_K$.}
\end{cases}
\end{eqnarray*}

\end{theorem}

\begin{proof}
Since $\ord_\ell(\Delta_K)\geq \ceil{\ord_\ell(\Delta_K)/2}$, applying Theorem \ref{geom1} to $KL$ gives
\begin{eqnarray}
\Br(\overline{E}\times\overline{E})_{\ell^{\infty}}^{\Gamma_{KL}}&=&
\Br(\overline{E}\times\overline{E})_{\ell^{k}}^{\Gamma_{KL}}=
(\End E_{\ell^k}/(\cO_K\otimes\bbZ/\ell^k))^{\Gamma_{KL}}\\
\label{size}
&\cong& \bbZ/\ell^{m+\floor{\ord_{\ell}(\Delta_K)/2}}\times \bbZ/\ell^{m+\ceil{\ord_{\ell}(\Delta_K)/2}}.
\end{eqnarray}
By part \ref{partclasstauKL} of Lemma \ref{tau fixed}, 
$$\cO_K\theta\subset (\End E_{\ell^k}/(\cO_K\otimes\bbZ/\ell^k))^{\Gamma_{KL}}.$$
Using part \ref{partplus} of Lemma \ref{tau fixed} to count the number of elements in $\cO_K\theta$ and comparing to \eqref{size} gives 
$$\cO_K\theta= (\End E_{\ell^k}/(\cO_K\otimes\bbZ/\ell^k))^{\Gamma_{KL}}.$$
Now part \ref{partclasstautau} of Lemma \ref{tau fixed} shows that 
$$\cO_{\ell^m}\theta= (\End E_{\ell^k}/(\cO_K\otimes\bbZ/\ell^k))^{\Gamma_{L}}.$$
Moreover, since $\ord_\ell(\Delta_K)\leq 1$ for an odd prime $\ell$, part \ref{partplus} of Lemma \ref{tau fixed} gives $\cO_{\ell^m}\theta\cong\bbZ/\ell^k$ if $\ell$ is odd or $\ell\nmid\Delta_K$. If $\ell=2$ and $2\mid\Delta_K$, then part \ref{partplus} of  Lemma \ref{tau fixed} gives $\cO_{2^m}\theta\cong\bbZ/2^{k-1}\times\bbZ/2$.
\end{proof}

\begin{theorem}
\label{quotient2}
Suppose that $K\nsubseteq L$ and let $\ell\in\bbZ_{>0}$ be prime. Let $m=m(\ell)$. Let $\eta$ denote the image of $\tau$ in the quotient group $\End E_{\ell^m}/(\cO_K\otimes\bbZ/\ell^m)$. Then
 \begin{eqnarray*}
 \left(\frac{\Br(E\times E)}{\Br_1(E\times E)}\right)_{\ell^{\infty}}= \frac{\Br(E\times E)_{\ell^{m}}}{\Br_1(E\times E)_{\ell^{m}}}=\frac{\End_{\Gamma_L}E_{\ell^m}}{(\cO_K\otimes\bbZ/\ell^m)^{\Gamma_L}}=(\bbZ/\ell^m)\eta \cong\bbZ/\ell^m
  \end{eqnarray*}
unless $\ell=2$, $2\mid\Delta_K$, $m\geq 1$ and $E_2=E_2(L)$, in which case
\begin{eqnarray*} \left(\frac{\Br(E\times E)}{\Br_1(E\times E)}\right)_{2^{\infty}}=\frac{\Br(E\times E)_{2^{m+1}}}{\Br_1(E\times E)_{2^{m+1}}}&=&\frac{\End_{\Gamma_L} E_{2^{m+1}}}{(\cO_K\otimes\bbZ/2^{m+1})^{\Gamma_L}}\\
&\cong & \bbZ/2^m\times\bbZ/2
\end{eqnarray*}
where the copy of $\bbZ/2^m$ is generated by the image of $\tau$.
\end{theorem}

\begin{proof}
Let $k=m+\ord_{\ell}(\Delta_K)$ and let $\theta$ denote the image of $\tau$ in the quotient group $\End E_{\ell^{k}}/(\cO_K\otimes\bbZ/\ell^{k})$. Then 
\begin{equation}
\label{quotient in geom}
\frac{\Br(E\times E)_{\ell^{\infty}}}{\Br_1(E\times E)_{\ell^{\infty}}}\hookrightarrow  \Br(\overline{E}\times\overline{E})_{\ell^{\infty}}^{\Gamma_{L}}=\cO_{\ell^m}\theta,
\end{equation}
by Theorem \ref{geom2}.
For all $t\in\bbZ_{\geq 0}$, 
\begin{equation}
\label{LfixedinKLfixed}
\frac{\Br(E\times E)_{\ell^t}}{\Br_1(E\times E)_{\ell^{t}}}=\frac{\End_{\Gamma_L}E_{\ell^t}}{(\cO_K\otimes\bbZ/\ell^t)^{\Gamma_L}}\hookrightarrow \frac{\End_{\Gamma_{KL}}E_{\ell^t}}{\cO_K\otimes\bbZ/\ell^t}.
\end{equation}
First suppose that $\ell$ is odd or $\ell\nmid\Delta_K$.
Then \eqref{quotient in geom} and \eqref{LfixedinKLfixed} combined with Theorems \ref{quotient1} and \ref{geom2} show that 
\begin{equation}
\label{ell^m}
\Bigl(\frac{\Br(E\times E)}{\Br_1(E\times E)}\Bigr)_{\ell^{\infty}}\hookrightarrow \bbZ/\ell^m.
\end{equation}

Consider $\tau$ as an element of $\End E_{\ell^m}$. By parts \ref{parttaufixKL} and \ref{parttaufixtau} of Lemma \ref{tau fixed}, \mbox{$\tau\in \End_{\Gamma_L} E_{\ell^m}$.} By part \ref{partplus} of Lemma \ref{tau fixed}, $\eta$ has order $\ell^m$ in 
\[\End_{\Gamma_L} E_{\ell^{m}}/(\cO_K\otimes\bbZ/\ell^m)^{\Gamma_L}=\Br(E\times E)_{\ell^m}/\Br_1(E\times E)_{\ell^m}.\] 
Hence, by \eqref{ell^m},
$$(\bbZ/\ell^m)\eta= \frac{\End_{\Gamma_L} E_{\ell^{m}}}{(\cO_K\otimes\bbZ/\ell^m)^{\Gamma_L}}=\Bigl(\frac{\Br(E\times E)}{\Br_1(E\times E)}\Bigr)_{\ell^{\infty}}.$$

Now suppose that $\ell=2$ and $2\mid\Delta_K$.
If $m(2)=0$, then $(\Br(E\times E)/\Br_1(E\times E))_{2^{\infty}}=0,$ by \eqref{LfixedinKLfixed} and Theorem \ref{quotient1} applied to $KL$. So 
we assume from now on that $m=m(2)\geq 1$.
Theorems \ref{quotient1} and \ref{geom2} combined with \eqref{quotient in geom} and \eqref{LfixedinKLfixed} show that
\begin{equation} 
\label{2power}
\left(\frac{\Br(E\times E)}{\Br_1(E\times E)}\right)_{2^{\infty}}\hookrightarrow \bbZ/2^m\times\bbZ/2.
\end{equation}
By parts \ref{partplus}, \ref{parttaufixKL} and \ref{parttaufixtau} of Lemma \ref{tau fixed}, the image of $\tau$ generates a copy of $\bbZ/2^m$ inside  
$\End_{\Gamma_L} E_{2^{m+1}}/(\cO_K\otimes\bbZ/2^{m+1})^{\Gamma_L}=\Br(E\times E)_{2^{m+1}}/\Br_1(E\times E)_{2^{m+1}}.$ 
Therefore, \eqref{2power} shows that $(\Br(E\times E)/\Br_1(E\times E))_{2^{\infty}}$ is isomorphic to either $\bbZ/2^m$ or $\bbZ/2^m\times\bbZ/2$.

First suppose that $E_2=E_2(L)$. Then $\Gamma_L$ acts trivially on $E_2$ and hence
$$\frac{\Br(E\times E)_2}{\Br_1(E\times E)_2}=\frac{\End_{\Gamma_L} E_2}{(\cO_K\otimes\bbZ/2)^{\Gamma_L}}=\frac{\End E_2}{\cO_K\otimes\bbZ/2}\cong\bbZ/2\times\bbZ/2.$$
Therefore, 
$$\left(\frac{\Br(E\times E)}{\Br_1(E\times E)}\right)_{2^{\infty}}=\frac{\Br(E\times E)_{2^{m+1}}}{\Br_1(E\times E)_{2^{m+1}}}\cong\bbZ/2^m\times\bbZ/2.$$

Now suppose that $E_2\neq E_2(L)$. By Theorem \ref{geom2}, $$ \Br(\overline{E}\times\overline{E})_{2^{\infty}}^{\Gamma_L}
=\left(\frac{\End E_{2^{k}}}{\cO_K\otimes\bbZ/2^{k}}\right)^{\Gamma_L}=\cO_{2^m}\theta$$
and, in particular, for any $t\in\bbZ_{\geq 0}$ the natural injection  
\begin{equation}
\label{inj=isom}
\cO_{2^m}\theta=\left(\frac{\End E_{2^{k}}}{\cO_K\otimes\bbZ/2^{k}}\right)^{\Gamma_L}\hookrightarrow\left(\frac{\End E_{2^{k+t}}}{\cO_K\otimes\bbZ/2^{k+t}}\right)^{\Gamma_L}
\end{equation}
induced by multiplication by $2^t$ on $E_{2^{k+t}}$ is an isomorphism. Let $\vphi\in\End_{\Gamma_L} E_{2^{k+t}}$ for some $t\in\bbZ_{\geq 0}$. We have 
\begin{equation}
\label{quotientingeom2}
 \frac{\End_{\Gamma_L} E_{2^{k+t}}}{(\cO_K\otimes\bbZ/2^{k+t})^{\Gamma_L}}\hookrightarrow
\left(\frac{\End E_{2^{k+t}}}{\cO_K\otimes\bbZ/2^{k+t}}\right)^{\Gamma_L}.
\end{equation}
 Since $2\mid\Delta_K$, we can write $\cO_K=\bbZ[\sqrt{-d}]$ where $\Delta_K=-4d$.
Since the injection in \eqref{inj=isom} is an isomorphism, we can use \eqref{quotientingeom2} to write 
\begin{equation}
\label{varphi}
\vphi=2^t(x+2^my\sqrt{-d})\tau+z+w\sqrt{-d}
\end{equation}
for some $x,y,z,w\in\bbZ/2^{k+t}$. Here we abuse notation slightly by using $\tau$ to denote the image of $\tau$ in $\End_{\Gamma_L} E_{2^{k+t}}$. 
Since $\vphi$ is fixed by $\tau$, we have
\[2\sqrt{-d}(2^{m+t}y\tau+w)\equiv 0\pmod{2^{k+t}}.\]
Multiplying by $\sqrt{-d}$ and recalling that $k=m+\ord_2(\Delta_K)=m+\ord_2(d)+2$, we see that
\[2^{m+t}y\tau+w\equiv 0\pmod{2^{m+t+1}}.\]
Therefore, $w=2^{m+t}u$ for some $u\in\bbZ/2^{k+t}$ and we have
\[y\tau+u\equiv 0\pmod{2}.\]
Suppose for contradiction that $y\not\equiv 0\pmod{2}$. Then $\tau$ acts as multiplication by a scalar on $E_2$. Furthermore, since $\tau$ is invertible, this scalar cannot be zero and therefore must be $1$. In other words, $\tau$ acts as the identity on $E_2$. Furthermore, since $m(2)\geq 1$, $\Gamma_{KL}$ acts trivially on $E_2$ and hence $E_2=E_2(L)$, giving the required contradiction. Therefore, $y\equiv 0\pmod{2}$ and we can write $y=2v$ for some \mbox{$v\in\bbZ/2^{k+t}$} and substituting into \eqref{varphi} gives
\begin{equation}
\label{cyclicclass}
\vphi=2^t(x+2^{m+1}v\sqrt{-d})\tau+z+w\sqrt{-d}.
\end{equation}
Now part \ref{partplus} of Lemma \ref{tau fixed} shows that $2^{t+m+1}\sqrt{-d}\tau\in\cO_K\otimes\bbZ/2^{k+t}.$ Thus, \eqref{cyclicclass} shows that the class of $\varphi$ in $(\End E_{2^{k+t}}/(\cO_K\otimes\bbZ/2^{k+t}))^{\Gamma_L}$ is represented by $2^tx\tau$. 
But $\varphi$ was arbitrary and \eqref{quotientingeom2} is injective, hence $\End_{\Gamma_L} E_{2^{k+t}}/(\cO_K\otimes\bbZ/2^{k+t})^{\Gamma_L}$ is a cyclic group. Therefore, 
$$\left(\frac{\Br(E\times E)}{\Br_1(E\times E)}\right)_{2^{\infty}}=\frac{\Br(E\times E)_{2^{m+1}}}{\Br_1(E\times E)_{2^{m+1}}}\cong\bbZ/2^m.$$
\end{proof}

\section{Special cases and examples}
We retain the notation and conventions of Section \ref{compute}. In particular,  $L$ is a number field and $E/L$ is an elliptic curve with complex multiplication by $\cO_K$.

\label{examples}

\begin{theorem}
\label{LinK}
Suppose that $L\subset H_K$, where $H_K$ denotes the Hilbert class field of $K$. Let $\ell\in\bbZ_{>0}$ be prime. Then $m(\ell)=n(\ell)=0$, except in the following special cases where $n(\ell)=1$:
\begin{enumerate}
\item $K=\bbQ(\zeta_3)$ and $\ell\leq 3$,
\item $K=\bbQ(i)$ and $\ell=2$,
\item $\Delta_K\equiv 1\pmod{8}$ and $\ell=2$.
\end{enumerate}
Consequently, if 
$\cO_K^*=\{\pm 1\}$ and $\Delta_K\not\equiv 1\pmod{8}$, then 
\[\Br(E\times E)=\Br_1(E\times E).\]

\end{theorem}

\begin{proof}
Let $j(E)$ denote the $j$-invariant of the elliptic curve $E$. Since $E$ is defined over $L$, we have $\bbQ(j(E))\subset L$. The theory of complex multiplication tells us that $K(j(E))=H_K$. Therefore, $[KL:K]=[H_K:K]=h(\cO_K)$. Using the formula for the degree of a ring class field, as given in \eqref{degreeformula}, we see that in every case, $[K_{\ell^2}:K]>h(\cO_K)$ so $n(\ell)\leq 1$.
Furthermore, $[K_{\ell}:K]>h(\cO_K)$ except in the special cases (i), (ii) and (iii) of the theorem. The rest follows immediately from Proposition \ref{upper bound} and Theorems \ref{quotient1} and \ref{quotient2}.
\end{proof}

\begin{remark}
Since $K(j(E))=H_K$, the hypothesis $L\subset H_K$ holds precisely when $L=H_K$ or \mbox{$L=\bbQ(j(E))$}. 
\end{remark}

If $\cO_K^*=\{\pm1\}$, then Proposition \ref{upper bound} allows us to calculate $m(\ell)$ for all primes $\ell\in\bbZ_{>0}$, and hence compute the transcendental part of $\Br(E\times E)$. On the other hand, if $K\in\{\bbQ(i),\bbQ(\zeta_3)\}$, then Proposition~\ref{upper bound} only tells us that $m(\ell)\leq n(\ell)$ for all primes $\ell\in\bbZ_{>0}$. The following two propositions deal with $K=\bbQ(i)$ and $K=\bbQ(\zeta_3)$, and in each case give sufficient conditions which allow us to conclude that $m(\ell)=0$.

\begin{proposition}
\label{d=1}
Let $\ell\in\bbZ_{>0}$ be an odd prime. Let $K=\bbQ(i)$.  
Suppose that there exists a finite prime $\frq$ of $KL$ satisfying all of the following conditions.
\begin{enumerate}
\item $\frq$ is coprime to $2\ell$,
\item $E$ has good reduction at $\frq$,
\item \label{division} $f_{\frs/\frp}\mid f_{\frq/\frp}$, where $\frp=\frq\cap\cO_K$ and $\frs$ is a prime of $K_{2\ell}$ above $\frp$,
\item $\psi_{E/KL}(\frq)\notin\cO_2$. 
\end{enumerate} 
Then $m(\ell)=0$, and hence \[(\Br(E\times E)/\Br_1(E\times E))_{\ell^{\infty}}=\Br(\overline{E}\times\overline{E})_{\ell^{\infty}}^{\Gamma_{L}}
=\Br(\overline{E}\times\overline{E})_{\ell^{\infty}}^{\Gamma_{KL}}=0.\]

\end{proposition}
Note that condition \ref{division} is trivially satisfied if $K_{2\ell}\subseteq KL$.

\begin{proof}
Let $\frq$ be a finite prime of $KL$ satisfying conditions (1)--(4). Let $\frp$ and $\frs$ be primes as described in condition \ref{division}.
The Artin symbol $(\frp, K_{2\ell}/K)$ has order $f_{\frs/\frp}$ in $\Gal(K_{2\ell}/K)$. 
 Since $f_{\frs/\frp}$ divides $f_{\frq/\frp}$, we have
 \[1=(\frp,K_{2\ell}/K)^{f_{\frq/\frp}}
 =(\frp^{f_{\frq/\frp}},K_{2\ell}/K)
 =(N_{KL/K}(\frq),K_{2\ell}/K).\]

By the definition of the ring class field $K_{2\ell}$, this implies that $$N_{KL/K}(\frq)=(\alpha)$$ for some $\alpha\in\cO_{2\ell}$. Now $\psi_{E/KL}(\frq)$ is a generator of $N_{KL/K}(\frq)$ but $\psi_{E/KL}(\frq)\notin\cO_2$ by the hypothesis, so $\psi_{E/KL}(\frq)=\pm i\alpha$. Therefore, $\psi_{E/KL}(\frq)\notin\cO_{\ell}$, and hence $m(\ell)=0$. 
\end{proof}

\begin{proposition}
\label{d=3}
Let $K=\bbQ(\zeta_3)$ and let $\ell\in\bbZ_{>0}$ be prime with $\ell\neq 3$. 
Suppose that there exists a finite prime $\frq$ of $KL$ satisfying all of the following conditions.
\begin{enumerate}
\item $\frq$ is coprime to $3\ell$ 
\item $E$ has good reduction at $\frq$,
\item $f_{\frs/\frp}\mid f_{\frq/\frp}$, where $\frp=\frq\cap\cO_K$ and $\frs$ is a prime of $K_{3\ell}$ above $\frp$,
\item $\psi_{E/KL}(\frq)\notin\cO_3$. 
\end{enumerate}
Then $m(\ell)=0$ and hence \[(\Br(E\times E)/\Br_1(E\times E))_{\ell^{\infty}}=\Br(\overline{E}\times\overline{E})_{\ell^{\infty}}^{\Gamma_{L}}
=\Br(\overline{E}\times\overline{E})_{\ell^{\infty}}^{\Gamma_{KL}}=0.\]

\end{proposition}
As before, condition \ref{division} is trivially satisfied if $K_{3\ell}\subseteq KL$.
\begin{proof}
The strategy is the same as for Proposition \ref{d=1}. 
\end{proof}

\begin{example}
Let $E$ be the elliptic curve over $\bbQ$ with affine equation
\[y^2+y=x^3-x^2-7x+10.\]
By \cite{LMFDB-11}, $E$ has complex multiplication by the ring of integers of $K=\bbQ(\sqrt{-11})$.
Theorem \ref{LinK} tells us that $m(\ell)=n(\ell)=0$ for every prime $\ell\in\bbZ_{>0}$ and therefore
\[\Br(E\times E)=\Br_1(E\times E).\]
Let $\theta$ denote the image of complex conjugation in $\End E_{11}/(\cO_K\otimes\bbZ/11)$. Then Theorem \ref{geom2} gives
\[\Br(\overline{E}\times\overline{E})^{\Gamma_{\bbQ(\sqrt{-11})}}
=\Br(\overline{E}\times\overline{E})^{\Gamma_{\bbQ}}
=\cO_K\theta\cong\bbZ/11.\]
\end{example}

\begin{example}
Let $E$ be the elliptic curve over $\bbQ$ with affine equation $$y^2=x^3-Dx$$ where $D\in\bbZ\setminus\{0\}$. Then $\End E=\bbZ[i]$. Let $K=\bbQ(i)$. For any odd prime $\ell\in\bbZ_{>0}$, Theorem~\ref{LinK} gives
\[(\Br(E\times E)/\Br_1(E\times E))_{\ell^{\infty}}=\Br(\overline{E}\times\overline{E})_{\ell^{\infty}}^{\Gamma_{\bbQ}}=\Br(\overline{E}\times\overline{E})_{\ell^{\infty}}^{\Gamma_{K}}=0.\]
Theorem \ref{LinK} tells us that $n(2)=1$. We must compute $m(2)$. 
By Proposition \ref{upper bound}, $m(2)\leq n(2)$. Let $\frq$ be a finite prime of $\bbZ[i]$ that is coprime to $2D$. Let $\pi_\frq\in\bbZ[i]$ be the unique generator of $\frq$ such that \mbox{$\pi_\frq\equiv 1\pmod{(2+2i)}$}. Exercise 2.34 in \cite{Silverman} shows that
\[\psi_{E/K}(\frq)=\Biggl(\frac{D}{\pi_\frq}\Biggr)^{-1}_4\pi_\frq\]
where $(\frac{\cdot}{\cdot})_4$ denotes the quartic residue symbol on $\bbZ[i]$.

First suppose that $D$ is a square in $\bbZ[i]$. Then for all finite primes $\frq$ which are coprime to $2D$, $\psi_{E/K}(\frq)=\pm\pi_\frq\in\cO_2$ and therefore $m(2)=1$. Let $\theta$ denote the image of complex conjugation in $\End E_8/(\bbZ[i]\otimes\bbZ/8)$. Applying Theorems \ref{geom2} and \ref{geom1}, we see that 
\begin{eqnarray*}
\Br(\overline{E}\times\overline{E})^{\Gamma_{K}}&=&\Br(\overline{E}\times\overline{E})_{2^{\infty}}^{\Gamma_{K}}
=\bbZ[i]\theta \cong \bbZ/4\times \bbZ/4\\
\textrm{and }\ \Br(\overline{E}\times\overline{E})^{\Gamma_{\bbQ}}&=&\Br(\overline{E}\times\overline{E})_{2^{\infty}}^{\Gamma_{\bbQ}}
=\cO_2\theta \cong \bbZ/4\times\bbZ/2.
\end{eqnarray*}
Applying Theorem \ref{quotient2}, we see that
\begin{eqnarray*}
\frac{\Br(E\times E)}{\Br_1(E\times E)}&=&\frac{\Br(E\times E)_4}{\Br_1(E\times E)_4}=\frac{\End_{\Gamma_{\bbQ}} E_4}{(\bbZ[i]\otimes\bbZ/4)^{\Gamma_{\bbQ}}}\\
&\cong & \begin{cases}\bbZ/2\times\bbZ/2 & \textrm{if $D$ is a square in $\bbZ$}\\
\bbZ/2 & \textrm{if $D$ is not a square in $\bbZ$.}
\end{cases}
\end{eqnarray*}

Now suppose that $D$ is not a square in $\bbZ[i]$. By \cite{CF}, Exercise 6.1, there exist infinitely many finite primes $\frq$ of $K$ coprime to $2D$ such that $D$ is not a square modulo $\frq$. For such $\frq$, we have $\psi_{E/K}(\frq)=\pm i\pi_\frq$ and therefore $\psi_{E/K}(\frq)\notin\cO_2$. Consequently, $m(2)=0$. Let $\eta$ denote the image of complex conjugation in $\End E_4/(\bbZ[i]\otimes\bbZ/4)$. Then Theorem \ref{geom2} gives
\[\Br(\overline{E}\times\overline{E})^{\Gamma_{K}}
=\Br(\overline{E}\times\overline{E})^{\Gamma_{\bbQ}}
=\bbZ[i]\eta\cong\bbZ/2\times \bbZ/2\]
and Theorem \ref{quotient2} gives
$\Br(E\times E)=\Br_1(E\times E).$
\end{example}

\begin{example}
\label{egzeta3}
Let $E$ be the elliptic curve over $\bbQ$ with affine equation $$y^2=x^3+D$$ where $D\in\bbZ\setminus\{0\}$. Then $\End E=\bbZ[\zeta_3]$, where $\zeta_3$ denotes a primitive $3$rd root of unity. Let $K=\bbQ(\zeta_3)$. For any prime $\ell>3$, Theorem \ref{LinK} tells us that $m(\ell)=0$ and therefore
\[(\Br(E\times E)/\Br_1(E\times E))_{\ell^{\infty}}=\Br(\overline{E}\times\overline{E})_{\ell^{\infty}}^{\Gamma_{\bbQ}}
=\Br(\overline{E}\times\overline{E})_{\ell^{\infty}}^{\Gamma_{K}}=0.\]
It remains to compute $m(\ell)$ for $\ell\leq 3$. 
For $\ell\leq 3$, Theorem \ref{LinK} gives $m(\ell)\leq 1$.
Let $\frq$ be a finite prime of $K$ that is coprime to $6D$. Let $\pi_\frq\in\bbZ[\zeta_3]$ be the unique generator of $\frq$ which satisfies $\pi_\frq\equiv 1\pmod{3}$. By \cite{Silverman}, Example II.10.6, the Gr\"{o}ssencharacter attached to $E/K$ is given by 
\begin{equation}
\label{grossencharacter}
\psi_{E/K}(\frq)=\Biggl(\frac{4D}{\pi_\frq}\Biggr)_6^{-1}\pi_\frq
\end{equation}
where $(\frac{\cdot}{\cdot})_6$ denotes the sextic residue symbol on $\bbZ[\zeta_3]$.

\smallskip

\paragraph{\underline{\emph{Computing $m(2)$}}}
By the cubic reciprocity law (see Theorem 7.8 of \cite{Lemmermeyer}, for example), 
\begin{equation}
\label{cubic reciprocity}
\Biggl(\frac{4}{\pi_\frq}\Biggr)_6=\Biggl(\frac{2}{\pi_\frq}\Biggr)_3=\Biggl(\frac{\pi_\frq}{2}\Biggr)_3\equiv \pi_\frq\pmod{2}
\end{equation}
where $(\frac{\cdot}{\cdot})_3$ denotes the cubic residue symbol on $\bbZ[\zeta_3]$. 
Substituting \eqref{cubic reciprocity} into \eqref{grossencharacter} gives
\begin{equation}
\label{mod2}
\psi_{E/K}(\frq)=\Biggl(\frac{4}{\pi_\frq}\Biggr)_6^{-1}\Biggl(\frac{D}{\pi_\frq}\Biggr)_6^{-1}\pi_\frq\equiv\Biggl(\frac{D}{\pi_\frq}\Biggr)_6^{-1}\pmod{2}.
\end{equation}
First, suppose that $D$ is a cube in $\bbZ$ (equivalently, $D$ is a cube in $\bbZ[\zeta_3]$). Then 
$\bigl(\frac{D}{\pi_\frq}\bigr)_6=\pm 1$
and \eqref{mod2} shows that $\psi_{E/K}(\frq)\in\cO_2$ for all finite primes $\frq$ that are coprime to $6D$.
Therefore, $m(2)=1$.

Now suppose that $D$ is not a cube in $\bbZ$. By \cite{CF}, Exercise 6.1, there exists a finite prime $\frq$ of $K$ coprime to $6D$ such that $D$ is not a cube modulo $\frq$. For such $\frq$, $\bigl(\frac{D}{\pi_\frq}\bigr)_6\neq\pm 1$, and \eqref{mod2} shows that $\psi_{E/K}(\frq)\notin\cO_2$. Therefore, $m(2)=0$.



\smallskip

\paragraph{\underline{\emph{Computing $m(3)$}}}

First suppose that $4D$ is a cube in $\bbZ$. Then \eqref{grossencharacter} shows that for all finite primes $\frq$ which are coprime to $6D$, $\psi_{E/K}(\frq)=\pm\pi_\frq\in\cO_3$. Hence, $m(3)=1$. 

Now suppose that $4D$ is not a cube in $\bbZ$. By \cite{CF}, Exercise 6.1, there exists a finite prime $\frq$ of $K$ coprime to $6D$ such that $4D$ is not a cube modulo $\frq$. For such $\frq$, $\bigl(\frac{4D}{\frq}\bigr)_6\neq\pm 1$, whereby $\psi_{E/K}(\frq)\notin\cO_3$. Therefore, $m(3)=0$. 

\end{example}

\section{The Brauer group of $\Kum(E\times E)$}
\label{Kummer}

Let $L$ be a number field and let $E/L$ be an elliptic curve with complex multiplication by an order $\cO$ of an imaginary quadratic field $K$. Let $X=\Kum(E\times E)$ be the K3 surface which is the minimal desingularisation of the quotient of $E\times E$ by the involution $(P,Q)\mapsto (-P,-Q)$.

\begin{proposition}
\label{Br1trivial}
If $\Delta_K\equiv 1\pmod{4}$ and $2\nmid [\cO_K:\cO]$ then $$\Br_1(X)=\Br(L)$$ and therefore there is no algebraic Brauer-Manin obstruction to weak approximation on $X$.
\end{proposition}

\begin{proof}
By Proposition 1.4 of \cite{SZtorsion}, it suffices to show that $H^1(L,\cO)=0$. Inflation-restriction gives
\begin{eqnarray*}
0\rightarrow H^1(\Gal(KL/L),\cO)\rightarrow H^1(L,\cO)\rightarrow H^1(KL,\cO)&=\Hom_{cts}(\Gamma_{KL},\bbZ^2)=0.
\end{eqnarray*}
Therefore, 
$H^1(L,\cO)\cong H^1(\Gal(KL/L),\cO)$. If $K\subset L$ then \mbox{$H^1(\Gal(KL/L),\cO)=0$,} so suppose that $$\Gal(KL/L)=\langle\tau\rangle\cong \bbZ/2.$$  Then 
$$H^1(\Gal(KL/L),\cO)=\frac{\{x\in\cO\mid x+\tau(x)=0\}}{\{\tau(x)-x\mid x\in\cO\}}.$$
Writing $\cO=\bbZ[f\alpha]$, where $f=[\cO_K:\cO]$ and $\alpha=(1+\sqrt{\Delta_K})/2$, gives
$$\{x\in\cO\mid x+\tau(x)=0\}=\{\tau(x)-x\mid x\in\cO\}=f\sqrt{\Delta_K}\cdot\bbZ.$$
\end{proof}

By \eqref{eq:embedding}, the existence of a transcendental element of odd order in $\Br(E\times E)$ implies that $\Br(X)$ contains a transcendental element. The same cannot be said for transcendental elements of even order. For this reason, we concentrate on elliptic curves $E$ for which $\Br(E\times E)$ contains a transcendental element of odd order.

\begin{theorem}
\label{oddtor}
Let $E/\bbQ$ be an elliptic curve with complex multiplication by $\cO_K$ such that $\Br(E\times E)$ contains a transcendental element of odd order. Then \mbox{$K=\bbQ(\zeta_3)$} and $E$ has affine equation $y^2=x^3+2c^3$ for some squarefree $c\in \bbZ$. Furthermore, 
$$\Br(X)/\Br(\mathbb{Q})=\Br(E\times E)/\Br_1(E\times E)=\Br(E\times E)_3/\Br_1(E\times E)_3=(\bbZ/3)\eta\cong\bbZ/3$$ where $\eta$ denotes the image of complex conjugation in $\End E_3/(\bbZ[\zeta_3]\otimes\bbZ/3)$.
\end{theorem}

\begin{proof}
Setting $L=\bbQ=\bbQ(j(E))$ in Theorem \ref{LinK} shows that $K=\bbQ(\zeta_3)$. Since $\bbZ[\zeta_3]$ has class number $1$, $E$ is isomorphic over $\overline{\bbQ}$ to the elliptic curve $E'$ with affine equation $y^2=x^3+1$. Therefore, $E$ is the sextic twist of $E'$ by a class in $H^1(\bbQ,\mu_6)=\bbQ^\times/(\bbQ^\times)^6$. Consequently, $E$ has an affine equation of the form $y^2=x^3+D$ for some sixth-power-free $D\in\bbZ$. Example~\ref{egzeta3} shows that $m(\ell)=0$ for every odd prime $\ell$ with $\ell\neq 3$. Since $\Br(E\times E)$ contains a transcendental element of odd order, we have $m(3)\neq 0$. The computation of $m(3)$ in Example~\ref{egzeta3} shows that $m(3)=1$ and $4D$ is a cube in $\bbZ$. Now the computation of $m(2)$ in Example~\ref{egzeta3} gives $m(2)=0$. Thus, the statement on the transcendental Brauer group of $E\times E$ follows from Theorem~\ref{quotient2}. The statement for $X=\Kum(E\times E)$ follows from \eqref{eq:embedding} and Proposition \ref{Br1trivial}.
\end{proof}

\section{A transcendental Brauer-Manin obstruction to weak approximation}\label{obstruction}

Henceforth, for each $c\in\bbQ^\times$, let $E^c$ be the elliptic curve over $\bbQ$ 
with affine equation 
\[y^2=x^3+2c^3.\]
Let $X=\Kum(E^c\times E^c)$. An affine model for $X$ is 
\begin{equation}
u^2=(x^3+2c^3)(t^3+2c^3)
\end{equation}
Note that $X$ is independent of $c\in\bbQ^\times$, since $(x,t,u)\mapsto (x/c,t/c,u/c^3)$ gives the following alternative affine model for $X$
\begin{equation}
\label{eq:Kummer}
u^2=(x^3+2)(t^3+2).
\end{equation}

By Proposition \ref{Br1trivial}, $\Br_1(X)=\Br(\bbQ)$ and therefore there is no algebraic Brauer-Manin obstruction to weak approximation on $X$. By \eqref{eq:embedding},
$$\Br(X)/\Br(\bbQ)=\Br(X)_3/\Br_1(X)_3=\Br(E^c\times E^c)_3/\Br_1(E^c\times E^c)_3.$$

Let $\tau\in\Gamma_{\bbQ}\setminus \Gamma_{\bbQ(\zeta_3)}$ and let $\theta$ denote the image of $\tau$ in $\End E^c_3$. 
The image of $\tau$ generates $\End_{\Gamma_{\bbQ}}(E^c_3)/(\bbZ/3)\cong\Br(X)/\Br(\bbQ)\cong\bbZ/3$. Let $\cA\in \Br(X)\setminus\Br(\bbQ)$ be a corresponding generator of $\Br(X)/\Br(\bbQ)$. 

For a prime $\ell$, let 
\[
\xymatrix{\cup: H^1(\bbQ_\ell, E^c_3)\times H^1(\bbQ_\ell,E^c_3)\ar[r] & \Br(\bbQ_\ell)_3\ar[r]^{\mathrm{inv}_\ell} &\frac{1}{3}\bbZ/\bbZ}
\]
be the non-degenerate pairing given by the composition of the cup product, the Weil pairing and the local invariant.
Let $\theta^*$ denote the map induced by $\theta$ on $H^1(\bbQ_\ell,E^c_3)$. For $P\in E(\bbQ_\ell)$, let $\chi_P$ denote the image of $P$ under the homomorphism 
$$\chi: E^c(\bbQ_\ell)\rightarrow H^1(\bbQ_\ell,E^c_3).$$ 

\begin{proposition}
\label{prop:cupproduct}
Let $P,Q\in E^c(\bbQ_\ell)\setminus E^c_2$. The $\bbQ_\ell$-point $(P,Q)$ on $E^c\times E^c$ gives rise to a point $R\in X(\bbQ_\ell)$. We have
\begin{equation}
\label{evaluation}
\ev_{\cA,\ell}(R)=\chi_P\cup\theta^*(\chi_Q)\in\frac{1}{3}\bbZ/\bbZ.
\end{equation}
\end{proposition}

\begin{proof}
The statement follows from the results of \cite{SZtorsion}, Section 3. The details are explained in Section 5.1 of \cite{I-S}.
\end{proof}

\begin{theorem}\label{thm5.2}
Let $\cA\in\Br(X)_3\setminus\Br(\bbQ)$.
Let $\nu\neq 3$ be a rational place. Then the evaluation map $\ev_{\cA,\nu}:X(\bbQ_{\nu})\rightarrow \Br(\bbQ_\nu)_3$ is zero. 
\end{theorem}

\begin{proof}
The statement for the infinite place is clear, since $\Br(\bbR)=\bbZ/2$ has trivial $3$-torsion. By \cite{goodred}, finite primes of good reduction do not appear in the description of the Brauer-Manin set. Lemma 4.2 of \cite{Matsumoto} shows that odd primes of good reduction for an abelian surface are primes of good reduction for the corresponding Kummer surface. Thus, by \eqref{eq:Kummer}, $\ev_{\cA,\ell}$ is zero for every finite prime $\ell\nmid 6$. From now on, let $\ell=2$, and let $R\in X(\bbQ_{2})$. We will show that $\ev_{\cA,2}(R)=0$. We can represent $R$ by $(x_0,t_0,u_0)$ satisfying \eqref{eq:Kummer}. Let $d_R=t_0^3+2$. Since the evaluation map $\ev_{\cA,2}:X(\bbQ_2)\rightarrow \Br(\bbQ_2)_3$ is locally constant, we are free to use the implicit function theorem to replace $R$ by a point $R'=(x_1,t_1,u_1)\in X(\bbQ_{2})$, sufficiently close to $R$, such that $d=d_{R'}\in\bbQ^\times$ and $u_1\neq 0$. Now $R'$ gives rise to $P=(dx_1,du_1)\in E^d(\bbQ_2)$ and $Q=(dt_1,d^2)\in E^d(\bbQ_2)$. Recalling that $X=\Kum(E^d\times E^d)$, we apply Proposition 
\ref{prop:cupproduct} to see that
\begin{equation}
\label{eq:cupprod2}
\ev_{\cA,2}(R')=\chi_P\cup\theta^*(\chi_Q)\in\frac{1}{3}\bbZ/\bbZ.
\end{equation}

We will study $E^d(\bbQ_2)$. Denote by $E^d_0(\bbQ_2)$ the $\bbQ_2$-points of $E^d$ that reduce to smooth points on the reduction of $E^d$ at $2$. Rescaling the coordinates $x,y$ of $E^d$ if necessary, we may assume that the $2$-adic valuation of $d$ satisfies $\ord_2(d)\in\{0,1\}$. 

First suppose that $\ord_2(d)=0$. Then the equation $y^2=x^3+2d^3$ is a minimal Weierstrass equation, and $E^d/\bbQ_2$ has additive reduction. Since $2\nmid d$, Tate's algorithm (as described in  \cite{Silverman}, Ch. IV, \S9) terminates at Step 3 and $E^d(\bbQ_2)=E^d_0(\bbQ_2)$. Alternatively, considering the $2$-adic valuations of the coordinates of points in $E^d(\bbQ_2)$ allows one to verify directly that $E^d(\bbQ_2)=E^d_0(\bbQ_2)$.

Now suppose that $d=2e$ for some $e\in\bbZ_2^*$. Then the equation $y^2=x^3+2d^3$ is no longer minimal. The transformation $y=y'+2^2$ gives the equation 
\begin{equation}
\label{eq:weierstrass}
y'^2+2^3y'=x^3+2^4(e^3-1).
\end{equation} 
Note that $e\in\bbZ_2^*$, so $e^3\equiv e\equiv \pm1\pmod{4}$. Suppose that $e\in - 1+4\bbZ_2$. Then we have $\ord_2(2^4(e^3-1))<6$, which implies that \eqref{eq:weierstrass} is minimal and $E^d/\bbQ_2$ has additive reduction. Applying Tate's algorithm, or considering the $2$-adic valuations of the coordinates of points in $E^d(\bbQ_2)$, we see that $E^d(\bbQ_2)=E^d_0(\bbQ_2)$.

In both the case $\ord_2(d)=0$ and the case $d\in -2+8\bbZ_2$, we may apply Theorem 1 of \cite{Rene} to see that $E^d(\bbQ_2)=E^d_0(\bbQ_2)$ is topologically isomorphic to $\bbZ_2$, which is \mbox{$3$-divisible}. Therefore, if $\ord_2(d)=0$ or $d\in -2+8\bbZ_2$ then $\chi=0$.

Finally, suppose that $d=2e$ with $e\in 1+4\bbZ_2$. 
Then \eqref{eq:weierstrass} is no longer minimal and the transformation $(x,y')=(2^2x'',2^3y'')$ gives  $y''^2+y''=x''^3+\frac{e^3-1}{4}$, which is a minimal Weierstrass equation for $E^d$. In particular, $E^d/\bbQ_2$ has good reduction. Therefore, 
\[E^d(\bbQ_2)/E^d_1(\bbQ_2)\cong\tilde{E^d}(\bbF_2)\cong\bbZ/3,\] 
where $E^d_1(\bbQ_2)$ denotes the kernel of the reduction map and $\tilde{E^d}$ denotes the reduced elliptic curve. Thus, $3E^d(\bbQ_2)\subset E^d_1(\bbQ_2)$. We will show that this inclusion is an equality. The standard filtration on the $\bbQ_2$-points of $E^d$ gives
$$E^d(\bbQ_2)\supset E^d_1(\bbQ_2)\supset E^d_2(\bbQ_2) \supset \dots $$
 The theory of formal groups shows that $E^d_2(\bbQ_2)\cong 4\bbZ_2$. Hence, $E^d_2(\bbQ_2)$ is $3$-divisible. Since $E^d_1(\bbQ_2) /E^d_2(\bbQ_2)\cong \bbZ/2$, it follows that $E^d_1(\bbQ_2)$ is $3$-divisible. Therefore, $$E^d_1(\bbQ_2)=3E^d_1(\bbQ_2)=3E^d(\bbQ_2).$$
Thus, $\chi$ factors through $E^d(\bbQ_2)/3E^d(\bbQ_2) = E^d(\bbQ_2)/E^d_1(\bbQ_2)\cong\bbZ/3$ and it is enough to show that 
\[\chi_P\cup\theta^*(\chi_P)=0\]
for any $P\in E^d(\bbQ_2)\setminus E^d_1(\bbQ_2)$ with $2P\neq 0$. The diagonal embedding $E^d\rightarrow E^d\times E^d$ induces a map $E^d\rightarrow X$ whose image is a copy of $\bbP^1_\bbQ$. The restriction of $\cA$ to $\bbP^1_{\bbQ}$ is in $\Br(\bbP^1_{\bbQ})=\Br(\bbQ)$. In other words, $\cA$ restricts to a constant algebra on the image of $E^d$ in $X$. Thus, the evaluation of $\cA$ at a point on $X$ corresponding to $(P,P)$ on $E^d(\bbQ_2)\times E^d(\bbQ_2)$ is independent of the point $P$. Hence, it suffices to show that $\chi_P\cup\theta^*(\chi_P)=0$ for a single $P\in E^d(\bbQ_2)$. Taking $P\in 3E^d(\bbQ_2)$ completes the proof.
\end{proof}

The main result of this section is the following theorem.
\begin{theorem}
\label{obstructionthm}
The evaluation map 
\[\ev_{\cA, 3}:X(\bbQ_3)\rightarrow \frac{1}{3}\bbZ/\bbZ\]
is surjective. Consequently, 
\[X(\bbA_{\bbQ})^{\Br(X)}=X(\bbQ_3)_0\times X(\bbR)\times \prod_{\ell\neq 3}{X(\bbQ_{\ell})}\ \subsetneq\  X(\bbA_{\bbQ})\]
where $X(\bbQ_3)_0$ denotes the points $P\in X(\bbQ_3)$ with $\ev_{\cA,3}(P)=0$, and the product runs over prime numbers $\ell\neq 3$.
\end{theorem}
Theorem \ref{obstructionthm} will be proved via several auxiliary results. 

\begin{lemma}
\label{lem:sufficient}
In order to show that $\ev_{\cA,3}:X(\bbQ_3)\rightarrow\frac{1}{3}\bbZ/\bbZ$ is surjective, it is enough to exhibit $c\in \bbQ^\times$ and $P\in E^c(\bbQ_3)$ such that $\theta^*(\chi_P)$ is not in the image of $E^c(\bbQ_3)$ inside $H^1(\bbQ_3, E^c_3)$.
\end{lemma}

\begin{proof}
Suppose that $P\in E^c(\bbQ_3)$ is such that $\theta^*(\chi_P)$ is not in the image of $E^c(\bbQ_3)$ inside $H^1(\bbQ_3, E^c_3)$. Since the image of $E^c(\bbQ_3)$ is a maximal isotropic subspace inside $H^1(\bbQ_3, E^c_3)$, there exists $Q\in E^c(\bbQ_3)$ such that $\chi_Q\cup\theta^*(\chi_P)\neq 0$. Note that $P,Q\notin E^c_2$ because if, for example, $2P=0$ then $\chi_P=\chi_{3P}=0$. Now by Proposition \ref{prop:cupproduct}, the point $R\in X(\bbQ_3)$ coming from $(Q,P)\in E^c(\bbQ_3)\times E^c(\bbQ_3)$ satisfies
\[\ev_{\cA,3}(R)=\chi_Q\cup\theta^*(\chi_P)\neq 0.\]
Surjectivity follows since for every $n\in\bbZ$, $\chi_{nQ}\cup \theta^*(\chi_P)=n(\chi_{Q}\cup \theta^*(\chi_P))$.
\end{proof}

In Proposition \ref{prop:3}, we will show that we can take $c=3$ and $P=(3,9)$ in Lemma \ref{lem:sufficient}. 
From now on, let $E=E^{(3)}$ be the elliptic curve with affine equation $y^2=x^3+2\cdot 3^3$. First, we determine the group $E(\bbQ_3)/3$ and give explicit generators. 

\begin{lemma}
\label{lem:generators}
We have
$E(\bbQ_3)/3\cong (\bbZ/3)^2$, with generators $P=(3,9)$ and $Q=(4,\sqrt{2\cdot 59})$.
\end{lemma}

\begin{proof}
Denote by $E_0(\bbQ_3)$ the $\bbQ_3$-points of $E$ that reduce to smooth points on the reduction of $E$ modulo $3$. Denote by $E_1(\bbQ_3)$ the kernel of reduction.
The elliptic curve $E/\bbQ_3$ has additive reduction and hence 
\begin{equation}
E_0(\bbQ_3)/E_1(\bbQ_3)\cong\bbF_3.
\end{equation}
Applying Tate's algorithm (or see the Tamagawa number for $3$ given in \cite{LMFDBj=0}), we find that 
\begin{equation}
\label{eq:c=3}
E(\bbQ_3)/E_0(\bbQ_3)\cong\bbZ/3.
\end{equation} 
 By Theorem 1 of \cite{Rene}, 
$E_0(\bbQ_3)\cong \bbZ_3$. The following sequence is exact.
\[
\xymatrix{
0\ar[r] & \frac{E_0(\bbQ_3)}{3E(\bbQ_3)}\ar[r]  & \frac{E(\bbQ_3)}{3E(\bbQ_3)}\ar[r] & \frac{E(\bbQ_3)}{E_0(\bbQ_3)}\ar[r] & 0.
}
\]
 Since $E_0(\bbQ_3)\cong\bbZ_3$ and $E_0(\bbQ_3)/E_1(\bbQ_3)\cong\bbF_3$, we have $3E_0(\bbQ_3)=E_1(\bbQ_3)$. By \eqref{eq:c=3}, $E(\bbQ_3)/E_0(\bbQ_3)\cong\bbZ/3$. A suitable generator is $P=(3,9)$. A calculation shows that \mbox{$3P=(3^{-2}\cdot 19, -3^{-3}\cdot 5\cdot 43)\in E_1(\bbQ_3)$}. Therefore, $3E(\bbQ_3)=E_1(\bbQ_3)$. The point $Q$ generates $E_0(\bbQ_3)/E_1(\bbQ_3)$.
\end{proof}

In light of Lemma \ref{lem:sufficient}, we will study the action of $\theta$ on the image of $E(\bbQ_3)$ in $H^1(\bbQ_3, E_3)$. We have 
\begin{equation*}
 E_3=\{O_E, (0,3\sqrt{6}),  (0,-3\sqrt{6})\}\cup\bigcup_{0\leq k\leq 2} \{(-6\zeta_3^k,9\sqrt{-2}), (-6\zeta_3^k,-9\sqrt{-2}) \}
\end{equation*}
and therefore
$$\mathbb{Q}_3(E_3) = \mathbb{Q}_3(\zeta_3,\sqrt{6},\sqrt{-2})= \mathbb{Q}_3(\zeta_3).$$

Let $F=\bbQ_3(E_3)=\bbQ_3(\zeta_3)$. The inflation-restriction exact sequence gives
\begin{eqnarray*}
H^1(\Gal(F/\bbQ_3), E_3)\rightarrow H^1(\bbQ_3,E_3)\rightarrow 
H^1(F,E_3)^{\Gal(F/\bbQ_3)}\rightarrow H^2(\Gal(F/\bbQ_3), E_3).
\end{eqnarray*}
 Since $[F:\bbQ_3]=2$, we have $H^1(\Gal(F/\bbQ_3), E_3)=H^2(\Gal(F/\bbQ_3), E_3)=0$. Therefore, the restriction map gives an isomorphism 
$$H^1(\bbQ_3,E_3)\rightarrow 
H^1(F,E_3)^{\Gal(F/\bbQ_3)}.$$
Let $T\in E(\bbQ_3)$. In a slight abuse of notation, we continue to write $\chi_T$ for the image of $T$ in \mbox{$H^1(F,E_3)=\Hom_{\mathrm{cts}}(\Gamma_F,E_3)$}. In order to study the action of $\theta$ on $\chi_T(\Gamma_F)\subset E_3$, we will use the following polynomials. Let $f_T\in\bbQ_3[t]$ be the degree $9$ polynomial satisfied by the $x$-coordinates of the points $R\in E(\overline{\bbQ_3})$ such that $3R=T$. By Exercise 3.7 of \cite{Silverman1}, 
\begin{equation}
\label{eq:f}
f_{T}(t)=3^2t^2(t-x(T))(t^3+2^3\cdot3^3)^2-2^3(t^3+2\cdot3^3)(t^6+2^3\cdot3^3\cdot5t^3-2^5\cdot3^6).
\end{equation}
Let $g_T\in\bbQ_3(\zeta_3)[t]$ be the cubic polynomial satisfied by the $x$-coordinates of the points \mbox{$S\in E(\overline{\bbQ_3})$} such that $(1-\zeta_3)S=T$. The addition formula shows that 
\begin{equation}
\label{eq:g}
g_T(t)=t^3+3\zeta_3x(T)t^2+2^3\cdot 3^3.
\end{equation}

 Combining Lemma \ref{lem:sufficient} with Proposition \ref{prop:3} below completes the proof of Theorem \ref{obstructionthm}.

\begin{proposition}
\label{prop:3}
Let $P=(3,9)\in E(\bbQ_3)$. Then $\theta^*(\chi_P)$ is not in the image of $E(\bbQ_3)$ inside $H^1(\bbQ_3,E_3)$. 
\end{proposition}

\begin{proof}
We have $\bbQ_3(E_3)=\bbQ_3(\zeta_3)$.
By Lemma \ref{lem:generators}, $E(\bbQ_3)/3$ is generated by $P=(3,9)$ and $Q=(4,\sqrt{2\cdot 59})$. A calculation using MAGMA \cite{Magma} shows that the degree $9$ polynomial $f_P$ given by \eqref{eq:f} is irreducible over $\bbQ_3$ and therefore also irreducible over $\bbQ_3(\zeta_3)$. 
By \eqref{eq:g}, we have
\begin{eqnarray*}g_P(t)=t^3+3^2\zeta_3t^2+2^3 \cdot 3^3 & \textrm{ and } & g_Q(t)=t^3+2^2 3\zeta_3t^2+2^3 \cdot 3^3.\end{eqnarray*}
Making a change of variables $t=3u$, we see that $g_Q(t)$ defines the same extension of $\bbQ(\zeta_3)$ as $h_Q(u)=u^3+2^2\zeta_3u^2+2^3$. Now $h_Q(u)\equiv u^3+u^2-1\pmod{(1-\zeta_3)}$, which is irreducible over the residue field $\bbF_3$ of $\bbQ_3(\zeta_3)$. Thus, $g_Q(t)$ defines an unramified extension of $\bbQ_3(\zeta_3)$. On the other hand, we claim that $g_P(t)$ defines a ramified extension of $\bbQ_3(\zeta_3)$.
Making a change of variables $t=3(u+1)$, we see that $g_P(t)$ defines the same extension of $\bbQ(\zeta_3)$ as
\mbox{$h_P(u)=u^3+3(1+\zeta_3)u^2+3(1+2\zeta_3)u+3\zeta_3 +3^2$}. Let $\pi=(1-\zeta_3)$. Examining the $\pi$-adic valuation of the terms in $h_P(u)$, we see that any root of $h_P(u)$ has $\pi$-adic valuation $2/3$. Therefore, $g_P(t)$ defines a ramified extension of $\bbQ_3(\zeta_3)$, as claimed.

Let $R_P, R_Q\in E(\overline{\bbQ_3})$ be such that $3R_P=P$ and $3R_Q=Q$. 
Let \mbox{$S_P=(1-\zeta_3^2)R_P$} and let $S_Q=(1-\zeta_3^2)R_Q$. Recall that $\bbQ_3(\zeta_3, x(R_P))$ is the degree $9$ extension of $\bbQ_3(\zeta_3)$ defined by $f_P$. Since $P$ is not a $2$-torsion point, $\bbQ_3(\zeta_3, x(R_P))=\bbQ_3(\zeta_3, R_P)$. Likewise, $g_P$ defines the ramified cubic extension $\bbQ_3(\zeta_3, S_P)/\bbQ_3(\zeta_3)$ and $g_Q$ defines the unramified cubic extension $\bbQ_3(\zeta_3, S_Q)/\bbQ_3(\zeta_3)$. 
Therefore, there exists $\sigma\in\Gamma_{\bbQ_3(\zeta_3)}$ such that $\sigma(S_Q)\neq S_Q$, $\sigma(S_P)=S_P$ and $\sigma(R_P)\neq R_P$. We have
\[(1-\zeta_3^2)\chi_P(\sigma)=(1-\zeta_3^2)(\sigma(R_P)-R_P)=\sigma(S_P)-S_P= 0\]
and
\[(1-\zeta_3^2)\chi_Q(\sigma)=(1-\zeta_3^2)(\sigma(R_Q)-R_Q)=\sigma(S_Q)-S_Q\neq 0.\]
Thus, $\chi_Q(\sigma)\notin E_{(1-\zeta_3)}$ and $\chi_P(\sigma)\in E_{(1-\zeta_3)}\setminus \{O_E\}$. Suppose for contradiction that $\theta^*(\chi_P)$ is in the image of $E(\bbQ_3)$ inside $H^1(\bbQ_3,E_3)$, so that
\begin{equation}
\label{eq:linearcombo}
\theta^*(\chi_P)=\chi_{(aP+bQ)}=a\chi_P+b\chi_Q
\end{equation}
for $a,b\in\bbF_3$. Note that  $\theta$ acts as multiplication by $-1$ on $E_{(1-\zeta_3)}=\{O_E, (0,3\sqrt{6}),  (0,-3\sqrt{6})\}$, so 
\begin{equation}
-\chi_P(\sigma)=\theta^*(\chi_P)(\sigma)=a\chi_P(\sigma)+b\chi_Q(\sigma)
\end{equation}
 which implies that $b\chi_Q(\sigma)\in E_{(1-\zeta_3)}$ and hence $b=0$ and $a=-1$. Since $g_P$ is irreducible over $\bbQ_3(\zeta_3)$, there exists $\rho\in \Gamma_{\bbQ(\zeta_3)}$ such that $\rho(S_P)\neq S_P$. For such $\rho$ we have 
\[(1-\zeta_3^2)\chi_P(\rho)=(1-\zeta_3^2)(\rho(R_P)-R_P)=\rho(S_P)-S_P\neq 0\]
and hence $\chi_P(\rho)\notin E_{(1-\zeta_3)}$. Therefore, $\chi_P(\Gamma_{\bbQ(\zeta_3)})=E_3$. In particular, $T=(-6\zeta_3,9\sqrt{-2})$ is in the image of $\chi_P$. But $\theta(T)\neq -T$, which contradicts $\theta^*(\chi_P)=-\chi_P$.
\end{proof}

\section*{Acknowledgements} I am very grateful to Alexei Skorobogatov for suggesting this problem, for several enlightening discussions and for pointing out a mistake in an earlier version of Theorem \ref{thm:Brauer-Maninset}. I would like to thank Dennis Eriksson, Paul Ziegler, Martin Bright, Spiros Adams-Florou, Jack Thorne and David Holmes for their enthusiasm and for some useful discussions. I am grateful to Peter Stevenhagen for his input which led towards the current formulation of Theorem \ref{LinK}. I would like to thank Tim Dokchitser and Srilakshmi Krishnamoorthy for some helpful comments on an earlier draft of this paper. Heartfelt thanks are due to the anonymous referee for many helpful and detailed comments which greatly improved the exposition. Most of this work was done during my stay at the Max Planck Institute for Mathematics in Bonn. I am grateful to the Max Planck Institute for financial support and for providing a very stimulating working environment.

\appendix

\section{}

This appendix contains a corrigendum to the main article above, said article having been published in \emph{J.~Lond.~Math.~Soc.} 93 (2016), 397--419.
In this corrigendum, we point out a mistake affecting the statements of Theorems~\ref{thm:Brauer-Maninset},~\ref{thm5.2}, \ref{obstructionthm} and Proposition~\ref{prop:cupproduct} and provide a correction.


\subsection{Errata}
The main article contains the following mistakes:
\begin{enumerate}
	\item\label{error} The statement of Theorem~\ref{thm:Brauer-Maninset} is incorrect: it does not hold as stated for all $\cA\in \Br(X)_3\setminus \Br(\bbQ)$. It holds for a particular choice of $\cA\in\Br(X)_3\setminus \Br(\bbQ)$, whose construction is explained in Section~\ref{sec:correct} below.
	The same is true of Proposition~\ref{prop:cupproduct}, Theorem~\ref{thm5.2} and Theorem~\ref{obstructionthm}. 
	 \item The proof of Lemma~\ref{lem:sufficient} uses the erroneous Proposition~\ref{prop:cupproduct}. 
	 While the statement of Lemma~\ref{lem:sufficient} does not technically need modification (see Remark~\ref{rem:5.4} below), it is cleaner to consider the whole of Section~\ref{obstruction} as pertaining to the particular element $\cA\in\Br(X)_3\setminus \Br(\bbQ)$ whose construction is explained in Section~\ref{sec:correct} of this corrigendum.
	 \item Typo: on p.5, in the sentence immediately preceding Definition~\ref{def:m}, `idele $(\dots , 1, 1, \pi_\mathfrak{q}, 1, 1, \dots) \in \mathbb{A}^\times_{KL}$' should read ` id\`{e}le $(1,\dots , 1, 1, \pi_\mathfrak{q}, 1, 1, \dots) \in \mathbb{A}^\times_{KL}$'.
	\item\label{typo} Typo: on p.6, in the proof of Proposition~\ref{upper bound}, the displayed equation `$\sigma(\alpha)\equiv \alpha^{N_{F/\bbQ}(\mathfrak{r})}\pmod{\mathfrak{s}}$' should read `$\sigma(\alpha)\equiv \alpha^{N_{N/\bbQ}(\mathfrak{r})}\pmod{\mathfrak{s}}$'. 
	\item Typo: on p.7, in Equation~\eqref{degreeformula}, `$h(O_K)$' should be replaced by `$h(\cO_K)$'.
	\item On p.8, below the fourth displayed equation, `In other words, for a prime $\mathfrak{q}$' should be replaced by `In other words, for some prime $\mathfrak{q}$'.
	\item Typo: all instances of $\theta^*$ in Section~\ref{obstruction} should be replaced by $\theta_*$.
	 \item On p.18, in the proof of Theorem~\ref{thm5.2}, the sentence `We can represent $R$ by $(x_0, t_0, u_0)$ satisfying~\eqref{eq:Kummer}' should be replaced by `Since the evaluation map $\ev_{\mathcal{A},2} : X (\bbQ_2) \to \Br(\bbQ_2)_3$ is locally constant, it is enough to show that it is zero on all $\bbQ_2$-points $R=(x_0, t_0, u_0)$ satisfying~\eqref{eq:Kummer}'.
\end{enumerate}
	
The typos are inconsequential but we mention them here for completeness.

\subsection{Corrections}\label{sec:correct}

Here we correct the significant mistake~\eqref{error} that necessitates this corrigendum.

In the statement of Theorem~\ref{thm:Brauer-Maninset}, the first sentence should read `Let $\cA\in\Br(X)_3\setminus \Br(\bbQ)$ be the element constructed in Section~\ref{obstruction}.'.

In Section~\ref{obstruction}, the sentence `Let $\cA\in\Br(X)\setminus\Br(\bbQ)$ be a corresponding generator of $\Br(X)/\Br(\bbQ)$'~should be replaced by the following paragraph:

`Following~\cite[\S5.1]{I-S},~\cite[\S3]{SZtorsion}, we construct an element $\cA\in\Br(X)_3$ from $\theta$. 
Multiplication by $3$ on $E^c$ turns $E^c$ into an $E^c$-torsor with structure group $E^c_3$. Denote this torsor by $\cT$ and let $[\cT]$ denote its class in $\mathrm{H}^1_{\et}(E^c, E^c_3)$. 
The automorphism $\theta\in\Aut_{\Gamma_{\bbQ}}(E^c_3)$ gives rise to the $E^c$-torsor $\theta_*\cT$ with structure group $E^c_3$. 
Composing the cup-product map with the Weil pairing $E^c_3\times E^c_3\to \mu_3$ yields a pairing
\begin{equation}\label{eq:torsorpair}
\mathrm{H}^1_{\et}(E^c\times E^c, E^c_3)\times\mathrm{H}^1_{\et}(E^c\times E^c, E^c_3)\to \Br(E^c\times E^c)_3.
\end{equation}
Let $p_1, p_2:E^c\times E^c\to E^c$ be the natural projection maps onto the first and second copies of $E^c$, respectively. The pullbacks $p_1^*\cT$ and $p_2^*\theta_*\cT$ are $E^c\times E^c$-torsors with structure group $E^c_3$; let $\cB\in \Br(E^c\times E^c)_3$ denote the pairing of their classes in $\mathrm{H}^1_{\et}(E^c\times E^c, E^c_3)$ via~\eqref{eq:torsorpair}. 
By the constructions of~\cite[Section~3 and Proposition~3.3]{SZtorsion}, the canonical map
\[\Br(E^c\times E^c)_3\to\End_{\Gamma_\bbQ}(E^c_3)/(\bbZ/3) \]
sends $\cB$ to the image of $\theta$.
Let $\iota$ denote the involution on $\Br(E^c\times E^c)$ induced by $(P,Q)\mapsto (-P,-Q)$ on $E^c\times E^c$. The proof of~\cite[Theorem~2.4]{SZtorsion}, in particular diagram~(16), identifies $\Br(X)$ with the subgroup of $\Br(E^c\times E^c)$ consisting of elements fixed by $\iota$. By the functoriality and bilinearity of the cup product, we find that $\iota(\cB)=\cB$. Let $\cA$ be the element of $\Br(X)$ corresponding to $\cB$'.

\bigskip

\begin{remark}\label{rem:5.4}
In the setting of Section~\ref{obstruction}, by Theorem~\ref{oddtor} we know that  $\Br(X)$ is generated by the element $\cA$ whose construction is described above and the image of $\Br(\bbQ)$. It follows from the corresponding results for $\cA$ that the statement of Lemma~\ref{lem:sufficient} holds for any element of $\Br(X)_3 \setminus \Br(\bbQ)$ and, moreover, that the following holds.

\begin{proposition}\label{prop:5.4}
Let $\cC \in \Br(X)_3 \setminus \Br(\bbQ)$. Let $v$ be a rational place. Then the evaluation map
$\ev_{\cC,v}: X(\bbQ_v) \to \frac{1}{3} \bbZ/\bbZ$ is surjective if $v = 3$ and constant otherwise.
\end{proposition}

\end{remark}

\begin{ack}
Rachel Newton was supported by UKRI Future Leaders Fellowship MR/T041609/1 and MR/T041609/2. Many thanks are due to Evis Ieronymou and the anonymous referee for their careful reading and useful comments which have improved this corrigendum.
\end{ack}


\end{document}